\documentclass[12pt, a4paper]{amsart}

\usepackage[hmargin=30mm, vmargin=25mm, includefoot, twoside]{geometry}
\usepackage[bookmarksopen=true]{hyperref}

\usepackage{amscd}
\usepackage{amsfonts,amssymb,verbatim}
\usepackage{latexsym}
\usepackage{mathrsfs}
\usepackage{stmaryrd}
\usepackage{xspace}
\usepackage{enumerate, paralist}
\usepackage{graphicx}
\usepackage[all]{xy}
\usepackage{extarrows}

\usepackage[noadjust]{cite}

\usepackage[usenames,dvipsnames]{color}

\usepackage{txfonts, pxfonts}

\usepackage{amsthm}
\usepackage{amsmath}

\usepackage{todonotes}
\numberwithin{equation}{section}

\newtheorem{thm}{Theorem}[section]
 \newtheorem{cor}[thm]{Corollary}
 \newtheorem{lem}[thm]{Lemma}
 \newtheorem{prop}[thm]{Proposition}
 
 \newtheorem{Q}[thm]{Question}
 
 \newtheorem{introthm}{Theorem}

\newtheorem{introcor}[introthm]{Corollary}

 \theoremstyle{definition}
  \newtheorem{defn}[thm]{Definition}

 \theoremstyle{remark}
 \newtheorem{rem}[thm]{Remark}

\newtheorem*{claim*}{Claim}

\def\NN{\mathbb{N}}

\def\CC{\mathbb{C}}

\def\G{\mathcal{G}}
\def\Gz{\mathcal{G}^{(0)}}

\def\H{\mathcal{H}}
\def\U{\mathcal{U}}

\def\B{\mathscr{B}}

\def\s{\mathrm{s}}
\def\r{\mathrm{r}}

\def\supp{\mathrm{supp}}

\def\Id{\mathrm{Id}}

\def\ker{\mathrm{Ker}}
\def\iso{\mathrm{Iso}}

\def\d{\mathrm{d}}
\def\fix{\mathrm{Fix}}
\def\triv{\mathrm{triv}}

\begin{document}

\title{Tracial states on groupoid $C^*$-algebras and essential freeness}
\author{Kang Li}
\author{Jiawen Zhang}

\address[Kang Li]{Department of Mathematics, Friedrich-Alexander-Universität Erlangen-Nürnberg,  Cauerstraße 11, 91058 Erlangen, Germany.}

\email{kang.li@fau.de}

\address[Jiawen Zhang]{School of Mathematical Sciences, Fudan University, 220 Handan Road, Shanghai, 200433, China.}
\email{jiawenzhang@fudan.edu.cn}

\thanks{JZ was partly supported by National Key R{\&}D Program of China 2022YFA100700.}

\thanks{Keywords: essential freeness, groupoid $C^*$-algebras, invariant measures, (quasidiagonal) traces.}

\begin{abstract}
Let $\G$ be a locally compact Hausdorff \'{e}tale groupoid. We call a tracial state $\tau$ on a general groupoid $C^*$-algebra $C_\nu^*(\G)$ \emph{canonical} if $\tau=\tau|_{C_0(\Gz)} \circ E$, where $E:C^*_\nu(\G) \to C_0(\Gz)$ is the canonical conditional expectation. In this paper, we consider so-called fixed point traces on $C_c(\G)$, and prove that $\G$ is essentially free if and only if any tracial state on $C_\nu^*(\G)$ is canonical and any fixed point trace is extendable to $C_\nu^*(\G)$. 

As applications, we obtain the following: 1) a group action is essentially free if every tracial state on the reduced crossed product is canonical and every isotropy group is amenable; 2) if the groupoid $\G$ is second-countable, amenable and essentially free then every (not necessarily faithful) tracial state on the reduced groupoid $C^*$-algebra is quasidiagonal.
\end{abstract}

\date{\today}

\maketitle

\section{introduction}\label{sec:introduction}

In recent years, there has been increasing interest in the classification of crossed products of $C^*$-algebras $C_0(X)\rtimes_{\nu} \Gamma$ arising from discrete amenable group actions on locally compact Hausdorff spaces $\Gamma \curvearrowright X$ (see, \emph{e.g.}, \cite{MR4534134, CLS21, MR3732883, MR4167017, KN21, MR4066584, LM23, MR3906305, MR4438064,N23,Niu19,MR4315611}). One of the key ingredients in those proofs is that every tracial state $\tau$ on $C_0(X)\rtimes_{\nu} \Gamma$ is \emph{canonical} in the sense that $\tau=\tau|_{C_0(X)} \circ E$, where $E:C_0(X)\rtimes_{\nu} \Gamma \to C_0(X)$ is the canonical conditional expectation. Actually, it was shown in \cite[Theorem~2.7]{KTT90} that every tracial state on the \emph{maximal} crossed product $C(X)\rtimes \Gamma$ of an action on a compact Hausdorff space $X$ is canonical if and only if the action is essentially free with respect to all invariant probability Radon measures on $X$. On the other hand, every tracial state on the \emph{reduced} crossed product $C(X)\rtimes_r \Gamma$ is canonical if and only if the action of the \emph{amenable radical} $R_a(\Gamma)$ of $\Gamma$ (\emph{i.e.}, the largest amenable normal subgroup in $\Gamma$) on $X$ is essentially free with respect to all invariant probability Radon measures on $X$ (see  \cite[Corollary 1.12]{Urs21}).

On the other hand, X. Li was able to show that all classifiable $C^*$-algebras
necessarily arise from twisted \'{e}tale groupoids (see \cite{MR4054809}). Hence, it is natural to consider canonical tracial states on general groupoid $C^*$-algebras $C_{\nu}^*(\G)$ of locally compact Hausdorff \'{e}tale groupoids $\G$. Similarly, a tracial state $\tau$ on $C_{\nu}^*(\G)$ is canonical if $\tau=\tau|_{C_0(\Gz)} \circ E$, where $E:C_{\nu}^*(\G) \to C_0(\Gz)$ is the canonical conditional expectation. If we let $\mu$ be the uniquely associated invariant probability Radon measure on $\Gz$ to $\tau|_{C_0(\Gz)}$, then it follows from \cite[Corollary~1.2]{Nes13} and \cite[Corollary 2.4]{NS22} that if $\G$ is second-countable, then a tracial state $\tau$ on the \emph{maximal} groupoid $C^*$-algebra $C^*(\G)$ is canonical if and only if $\G$ is \emph{essentially free with respect to the associated measure $\mu$} (see Definition~\ref{defn:ess free}). 

In this article, we would like to consider the relationship between the essential freeness of $\G$ and tracial states on a general groupoid $C^*$-algebra $C^*_\nu(\G)$ with respect to any $C^*$-norm $\|\cdot\|_\nu$ dominating the reduced $C^*$-norm. More precisely, we ask the following question:

%firstly generalize one direction of the previous results for groupoid $\G$ from the maximal $C^*$-norm on $C_c(\G)$ to any $C^*$-norm $\|\cdot\|_\nu$ dominating the reduced $C^*$-norm (see Proposition~\ref{prop:ess. free implies unique trace}), and then we would also like to consider the converse implication, namely the following question:

%(see \emph{e.g.}  \cite{MR2134336, MR4165470, ABBL20,MR4534134, Ma21, MW20})

\begin{Q}\label{Ques:converse}
Let $\G$ be a locally compact Hausdorff and \'{e}tale groupoid and $\|\cdot\|_\nu$ be a $C^*$-norm on $C_c(\G)$ dominating the reduced $C^*$-norm. If $\mu$ is an invariant probability Radon measure on $\Gz$, can we characterise the essential freeness of $\G$ with respect to $\mu$ in terms of tracial states on $C^*_\nu(\G)$?
\end{Q}

One of the crucial ingredients to answer Question \ref{Ques:converse} is the extendibility of the so-called \emph{fixed point trace} $\tau^\fix_\mu$ from $C_c(\G)$ to $C^*_\nu(\G)$ (see Definition~\ref{defn:fix trace}). If $\G=X\rtimes \Gamma$ is a transformation groupoid, then $\tau^\fix_\mu$ has a simplified form (see Equation (\ref{EQ:tau fix group action case})):
% from \cite{KTT90}:
 $$\tau^\fix_\mu(\sum_{i=1}^n f_i \gamma_i)= \sum_{i=1}^n \int_{\fix(\gamma_i)} f_i \d \mu,$$ where $f_i \in C_0(X)$ and $\gamma_i \in \Gamma$ for $i =1, \dotsc, n$, and $\fix(\gamma_i)=\{x\in X:\gamma_i x=x\}$ is the set of \emph{fixed points} of $\gamma_i$. The key observation is that $\tau^\fix_\mu$ defined on $C_c(\G)$ is canonical if and only if $\G$ is essentially free with respect to $\mu$ (see Lemma~\ref{lem:tau fix char strong eff.}). 
Using the fixed point trace, we prove the following main result of this paper which answers Question~\ref{Ques:converse}:

% As a consequence, it leads to the following main result of this paper, which answers Question~\ref{Ques:converse}:

\begin{introthm}(Theorem~\ref{thm:char for ess free max case})
Let $\G$ be a locally compact Hausdorff and \'{e}tale groupoid, and $\mu$ be an invariant probability Radon measure on $\Gz$. Then the following are equivalent:
\begin{itemize}
 \item[(1)] $\G$ is essentially free with respect to $\mu$;
 \item[(2)] For any $C^*$-norm $\|\cdot\|_\nu$ on $C_c(\G)$ dominating the reduced $C^*$-norm, any tracial state $\tau$ on $C^*_\nu(\G)$ with the associated measure being $\mu$ is canonical and  $\tau^\fix_\mu$ can be extended to a tracial state on $C^*_\nu(\G)$; 
 \item[(3)] There exists a $C^*$-norm $\|\cdot\|_\nu$ on $C_c(\G)$ dominating the reduced $C^*$-norm such that any tracial state $\tau$ on $C^*_\nu(\G)$ with the associated measure being $\mu$ is canonical and $\tau^\fix_\mu$ can be extended to a tracial state on $C^*_\nu(\G)$; 
 \item[(4)] Any tracial state $\tau$ on $C^*(\G)$ with the associated measure being $\mu$ is canonical.

\end{itemize}
\end{introthm}

On \emph{reduced} crossed products of $C^*$-algebras, the extendability of $\tau^\fix_\mu$ can be reformulated using a recent result in \cite{CN22} (see Proposition~\ref{prop:char for tau fix extension}) as follows:

\begin{introcor}(Corollary~\ref{cor:char for tau fix extension group action case})\label{cor:B}
Let $\Gamma$ be a discrete group acting on a locally compact Hausdorff space $X$ with an invariant probability Radon measure $\mu$. We consider the following conditions:
\begin{enumerate}
 \item The action is essentially free with respect to $\mu$;
 \item Any tracial state $\tau$ on $C_0(X) \rtimes_r \Gamma$ with the associated measure being $\mu$ is canonical;
 \item The isotropy group $\Gamma_x$ is amenable for $\mu$-almost every $x\in X$.
\end{enumerate}
Then (1) $\Rightarrow$ (2)\footnote{Note that the implication ``(2) $\Rightarrow$ (1)'' in Corollary~\ref{cor:B} does not hold in general (see Remark~\ref{rem:C*-simple}).} and (2) + (3) $\Rightarrow$ (1). 

If additionally $\Gamma$ is countable and $X$ is second-countable, then (1) $\Leftrightarrow$ (2) + (3).

\end{introcor}

We end this paper with the following result about quasidiagonal traces, which play a crucial role in the classification of simple nuclear $C^*$-algebras (see \cite{MR3966830}):

\begin{introcor}(Corollary~\ref{cor:amentrace}\footnote{We also refer the reader to Theorem~\ref{quasiThm} for a more general result.})
Let $\G$ be a locally compact, Hausdorff, second-countable, amenable and étale groupoid, which is also essentially free. Then every (not necessarily faithful) tracial state on $C_r^*(\G)$ is quasidiagonal.
\end{introcor}

%According to \cite[Lemma~4.3]{LR19}, we expect that the results in this paper also hold for \emph{twisted} \'{e}tale groupoids.  
\section{Preliminaries}\label{sec:Preliminaries}

%\subsection{Standard notation}\label{ssec:notation}

%Here we collect the notation used throughout the paper.

%To simplify the terminology, we always assume that locally compact spaces are Hausdorff. 

%For a subset $A$ in a topological space $X$, denote by $\overline{A}$ the closure of $A$ and by $A^\circ$ the interior of $A$.

Given a locally compact Hausdorff space $X$, we denote by $C(X)$ the set of complex-valued continuous functions on $X$. Recall that the \textit{support} of a function $f\in C(X)$ is the closure of $\{ x\in X: f(x)\neq 0\}$, written as $\supp f$. 
Denote by $C_c(X)$ the set of complex-valued continuous functions with compact support, and by $C_0(X)$ the set of complex-valued continuous functions vanishing at infinity, which is the closure of $C_c(X)$ with respect to the supremum norm $\|f\|_\infty:=\sup\{|f(x)|: x\in X\}$.

%For a $C^*$-algebra $A$ and a subset $S \subseteq A$, denote $\langle S \rangle_A$ (or merely $\langle S \rangle$ if the ambient algebra is clear from the context) the ideal in $A$ generated by $S$. 

\subsection{Basic notions for groupoids}\label{ssec:groupoid}

Let us start with some basic notions and terminologies about groupoids. For details we refer the reader to \cite{Ren80,MR4321941}.

Recall that a \emph{groupoid} is a small category, in which every morphism is invertible. Roughly speaking, a groupoid consists of a set $\G$, a subset $\Gz$ called the \emph{unit space}, two maps $\s, \r: \G \to \Gz$ called the \emph{source} and \emph{range} maps respectively, a \emph{composition law}:
\[
    \G^{(2)}\coloneqq \{(\gamma_1,\gamma_2) \in \G \times \G: \s(\gamma_1)=\r(\gamma_2)\}\ni(\gamma_1,\gamma_2) \mapsto \gamma_1\gamma_2 \in \G,
\]
and an \emph{inverse} map on $\G$ given by $\gamma \mapsto \gamma^{-1}$. These operations satisfy a couple of axioms, including the associativity law and the fact that elements in $\Gz$ act as units. 
For $x \in \Gz$, we define $\G^x:=\r^{-1}(x)$ and $\G_x:=\s^{-1}(x)$. Moreover, $\G^{x}_{x} = \G^{x}\cap\G_{x}$ is called the \emph{isotropy group} at $x\in \Gz$. A subset $Y \subseteq \Gz$ is called \emph{invariant} if $\r^{-1}(Y)=\s^{-1}(Y)$, and we define $\G_Y:=\s^{-1}(Y)$. %A \emph{groupoid morphism} is a functor.
For $A, B \subseteq \G$, we define
\begin{align*}
    A^{-1} & \coloneqq \{ \gamma^{-1} \in\G: \gamma \in A \};                                                                                                  \\
    AB     & \coloneqq \{ \gamma\in\G: \gamma=\gamma_{1}\gamma_{2}\text{ where }\gamma_{1} \in A,\gamma_{2}\in B \text{ and } \s(\gamma_1) = \r(\gamma_2) \}.
\end{align*}
%We say that $A \subseteq \G$ is \emph{symmetric} if $A=A^{-1}$.

A \textit{locally compact Hausdorff groupoid} is a groupoid $\G$ endowed with a locally compact and Hausdorff topology for which the composition, inversion, source and range maps are continuous with respect to the induced topologies.

We say that a locally compact Hausdorff groupoid $\G$ is \textit{\'{e}tale} if the range (and hence the source) map is a local homeomorphism, \emph{i.e.}, for any $\gamma \in \G$ there exists an open neighbourhood $U$ of $\gamma$ such that $\r(U)$ is open and $\r|_U$ is a homeomorphism. In this case, the fibers $\G_x$ and $\G^x$ with the induced topologies are discrete for each $x\in \Gz$, and $\Gz$ is clopen in $\G$.

 A subset $A$ in an \'{e}tale groupoid $\G$ is called a \textit{bisection} if the restrictions of $\s, \r$ to $A$ are homeomorphisms onto their respective images. It follows from definitions that all open bisections form a basis for the topology of $\G$. As a direct consequence, any function $f \in C_c(\G)$ can be written as a linear combination of continuous functions whose supports are contained in pre-compact open bisections.
%Throughout the paper, we shall limit ourselves to the \'{e}tale case.

%We note that, $A$ is a bisection if and only if $AA^{-1}$ and $A^{-1}A$ are contained in $\Gz$.
%If $A$, $B$ are bisections then $AB$ and $A^{-1}$ are bisections.

We record the following known result (see, \emph{e.g.}, \cite[Lemma 8.4.11]{MR4321941}). For convenience of the reader, we provide here a short proof.
\begin{lem}\label{lem:multiplication of bisection}
Let $\G$ be a locally compact Hausdorff and \'{e}tale groupoid. Then the multiplication map $\G^{(2)} \to \G$ is open.
\end{lem}

\begin{proof}
Given two open subsets $U$ and $V$ of $\G$, we need to show that $UV$ is also open. Without loss of generality, we can assume that $U$ and $V$ are open bisections with $\s(U) = \r(V)$. Fix $\alpha \in U$ and $\beta \in V$ with $\s(\alpha) = \r(\beta)$. Since the multiplication map is continuous at $(\alpha^{-1}, \alpha \beta)$, there exists an open neighbourhood $U_1 \subseteq U$ of $\alpha$ and an open bisection $W \subseteq \r^{-1}(\r(U_1))$ containing $\alpha \beta$ such that $U_1^{-1}W \subseteq V$. As $U_1 \subseteq U$ is also a bisection, we obtain that $W \subseteq U V$. 
\end{proof}

\noindent\textbf{Convention: Throughout the paper, we always assume that $\G$ is a locally compact Hausdorff and \'{e}tale groupoid.}

\subsection{Groupoid $C^*$-algebras}\label{ssec:groupoid C*-algebra}

Let us now recall different constructions of groupoid $C^*$-algebras and their basic properties. Given a locally compact Hausdorff and \'{e}tale groupoid $\G$, the space $C_c(\G)$ can be turned into a $*$-algebra with the following operations: given $f,g\in C_c(\G)$, we define their \emph{convolution} and \emph{involution} by
\begin{align}
    (f*g)(\gamma) & := \sum_{\alpha\in \G_{s(\gamma)}}f(\gamma\alpha^{-1})g(\alpha), \label{EQ:convolution} \\
    f^*(\gamma)   & := \overline{f(\gamma^{-1})}. \label{EQ:star}
\end{align}
To tell the difference, we denote the point-wise product by $f \cdot g$.

Recall that for each $x\in\Gz$, the \emph{left regular representation at $x$}, denoted by $\lambda_x: C_c(\G) \to \B(\ell^2(\G_x))$, is defined as follows:
\begin{equation}\label{EQ:reduced algebra defn}
    \left(\lambda_x(f)\xi\right)(\gamma)=\sum_{\alpha\in \G_x}f(\gamma\alpha^{-1})\xi(\alpha), \quad \text{where $\gamma\in \G_x$, } f\in C_{c}(\G) \text{ and }\xi\in\ell^{2}(\G_{x}).
\end{equation}
It is routine to check that $\lambda_x$ is a well-defined $\ast$-homomorphism.
%We can show that, the map $\lambda_{x}(f)$ is bounded and $\|\lambda_{x}(f)\|\leq\|f\|_{1}$.
%Thus, for each $x\in\Gz$ we have a $*$-representation $\lambda_{x}: C_{c}(\G)\rightarrow \mathscr{B}(\ell^2(\G_x))$.
The \textit{reduced $C^*$-norm} on $C_{c}(\G)$ is defined by
\[
    \| f \|_{r}\coloneqq \sup_{x\in \Gz} \| \lambda_{x}(f) \|,
\]
and the \textit{reduced groupoid $C^{*}$-algebra} $C^{*}_r(\G)$ is defined to be the completion of the $*$-algebra $C_{c}(\G)$ with respect to the reduced $C^*$-norm $\| \cdot \|_{r}$. 
%Sometimes we simply write $\|\cdot\|$ instead of $\|\cdot\|_r$ without ambiguity. 
It is clear that each left regular representation $\lambda_x$ can be extended automatically to a $*$-homomorphism $\lambda_x: C^*_r(\G) \to \B(\ell^2(\G_x))$.

We also consider the following norm on $C_c(\G)$ defined by:
\[
\|f\|_I:=\max\Big\{\sup_{x \in \Gz} \sum_{\gamma \in \G_x} |f(\gamma)|, \sup_{x \in \Gz} \sum_{\gamma \in \G_x} |f^*(\gamma)|\Big\}.
\]
The completion of $C_c(\G)$ with respect to the norm $\|\cdot\|_I$ is denoted by $L^1(\G)$. Recall that the \emph{maximal groupoid $C^*$-algebra} $C^*(\G)$ is defined to be the completion of $C_c(\G)$ with respect to the $C^*$-norm:
$$\|f\|_{\max}:=\sup \|\pi(f)\|,$$
where the supremum is taken over all bounded $\ast$-representations $\pi$ of $L^1(\G)$. It is clear that there is a surjective $\ast$-homomorphism
\begin{equation*}
q_{\max}: C^*(\G) \longrightarrow C^*_r(\G),
\end{equation*}
which is the identity on $C_c(\G)$. We say that $\G$ has \emph{the weak containment property} if $q_{\max}$ is an isomorphism.

We will also consider other $C^*$-norms between the reduced and the maximal ones. More precisely, we say that a $C^*$-norm $\|\cdot\|_\nu$ \emph{dominates the reduced $C^*$-norm} if $\|f\|_\nu \geq \|f\|_r$ for all $f \in C_c(\G)$. It is worth noticing that $\|\cdot\|_\nu \leq \|\cdot\|_{\max}$ always holds. We denote the $C^*$-completion of $C_c(\G)$ with respect to $\|\cdot\|_\nu$ by $C^*_\nu(\G)$, called a \emph{groupoid $C^*$-algebra} of $\G$. Similarly, we have a surjective $\ast$-homomorphism
\begin{equation}\label{EQ:canonical quotient map general}
q_{\nu}: C^*_\nu(\G) \longrightarrow C^*_r(\G),
\end{equation}
which is the identity on $C_c(\G)$.

We remark that there is an inclusion map $\iota_0: C_c(\Gz) \to C_c(\G)$ given by extending functions by zero on $\G \setminus \Gz$, and it was recorded in \cite[Section 2.2]{BL20} that $\iota_0$ can be extended to an isometric $\ast$-homomorphism $\iota: C_0(\Gz) \hookrightarrow C^*_r(\G)$, where the norm on $C_0(\Gz)$ is the supremum norm. The same fact holds for any $C^*$-norm $\|\cdot\|_\nu$ dominating the reduced $C^*$-norm. Hence, we will in what follows regard $C_0(\Gz)$ as a $C^*$-subalgebra in $C^*_\nu(\G)$ without further explanation.

%Concerning the maximal norm, we record the following well-known result. The proof is straightforward, and hence omitted.
%
%
%\begin{lem}\label{lem:C0(Gz) in Cmax}
%On $C_c(\Gz)$, the reduced norm coincides with the maximal norm. Hence $C_0(\Gz)$ can also be regarded as a $C^\ast$-subalgebra in $C^*(\G)$, and the quotient map $q$ from (\ref{EQ:canonical quotient map}) is the identity map when restricted on $C_0(\Gz)$.
%\end{lem}

From \cite[Proposition II.4.2]{Ren80} (see also \cite[Section 2.2]{BCS22}) we have that any element of $C^*_r(\G)$ can be regarded as a $C_0$-function on the groupoid $\G$. Indeed, there exists a linear and contractive map $j: C^*_r(\G) \to C_0(\G)$ given by
\[
j(a)(\gamma):=\left\langle \lambda_{\s(\gamma)}(a)\delta_{\s(\gamma)}, \delta_{\gamma} \right\rangle_{\ell^2(\G_{\s(\gamma)})}
\]
for $a\in C^*_r(\G)$ and $\gamma \in \G$. On $C_c(\G) \cup C_0(\Gz)$ the map $j$ is nothing but the identity map. 
%Renault showed that for $a,b\in C^*_r(\G)$, the convolution formula for $j(a) \ast j(b)$ is a convergent series that converges to $j(a \ast b)$. 
%For simplicity, we denote $(j(a) \neq 0):=\{\gamma \in \G: j(a)(\gamma) \neq 0\}$ for $a\in C^*_r(\G)$.
The reduced groupoid $C^*$-algebra also admits a faithful conditional expectation $E:C^*_r(\G) \to C_0(\Gz)$ defined by
\begin{equation}\label{EQ:expectation}
E(a)(u):=\left\langle \lambda_{u}(a)\delta_u, \delta_u \right\rangle_{\ell^2(\G_u)}
\end{equation}
for $a\in C^*_r(\G)$ and $u\in \Gz$ (see, \emph{e.g.}, \cite[Section 2.2]{BL20}). Intuitively, $E$ is given by restriction of functions in the sense that $j(E(a)) = j(a)|_{\Gz}$ for all $a\in C^*_r(\G)$. Hence, it follows that $E \circ \iota = \Id_{C_0(\Gz)}$. For any $C^*$-norm $\|\cdot\|_\nu$ dominating the reduced $C^*$-norm, we can compose $E$ with $q_\nu$ and obtain a conditional expectation $E \circ q_\nu: C^*_\nu(\G) \to C_0(\Gz)$ on $C^*_\nu(\G)$.

We end this subsection with an elementary fact, and leave its proof (which is a relatively straightforward computation) to the reader.

\begin{lem}\label{lem:easy calculation 0}
For $a\in C^*_r(\G)$, $f,g \in C_0(\Gz)$ and $\gamma \in \G$, we have 
\[
j(f \ast a \ast g)(\gamma) = f(\r(\gamma)) \cdot j(a)(\gamma) \cdot g(\s(\gamma)).
\] 
\end{lem}

\subsection{Tracial states}\label{ssec:tracial states}

Let $\G$ be a locally compact Hausdorff and \'{e}tale groupoid, and $\|\cdot\|_\nu$ be a $C^*$-norm on $C_c(\G)$ dominating the reduced $C^*$-norm. 

\begin{defn}\label{defn:tr st}
A \emph{tracial state} on the groupoid $C^*$-algebra $C^*_\nu(\G)$ is a state $\tau: C^*_\nu(\G) \to \CC$ satisfying $\tau(ab) = \tau(ba)$ for any $a,b \in C^*_\nu(\G)$. 
%Similarly, we also consider tracial states on any groupoid $C^*$-algebra $C^*_\nu(\G)$ with respect to some norm $\|\cdot\|_\nu$ dominating the reduced norm.
\end{defn}

If $\tau$ is a tracial state on the groupoid $C^*$-algebra $C^*_\nu(\G)$, then $\tau|_{C_0(\Gz)}$ is a state on $C_0(\Gz)$, which corresponds to a (positive) probability Radon measure $\mu$ on $\Gz$ according to the Riesz representation theorem. In other words, we have
\begin{equation}\label{EQ:canonical trace}
\tau(f) = \int_{\Gz} f \d\mu \quad \text{for any} \quad f \in C_0(\Gz).
\end{equation}
We call $\mu$ the \emph{measure associated} to $\tau$ and we also denote this measure by $\mu_\tau$. It is actually invariant in the following sense (see, \emph{e.g.}, \cite[Lemma 4.1]{LR19}): 

Given a bisection $B \subseteq \G$, we consider the homeomorphism 
\begin{equation}\label{EQ:alpha B}
\alpha_B:\s(B) \to \r(B) \quad \text{given  by} \quad x \mapsto \r((\s|_B)^{-1}(x)) \quad \text{for} \quad x \in \s(B).
\end{equation}
A Borel measure $\mu$ on $\Gz$ is called \emph{invariant} (cf. \cite[Definition I.3.12]{Ren80}) if for any open bisection $B$ in $\G$, we have $\mu|_{\r(B)} = (\alpha_B)_\ast(\mu|_{\s(B)})$. For an invariant measure $\mu$ on $\Gz$, its support $\supp\mu$ is an invariant subset of $\Gz$. We also define
\begin{equation}\label{EQ:fixed points}
\fix (\alpha_B) :=\{x\in \s(B): \alpha_B(x)=x\},
\end{equation}
which is an intersection of an open set and a closed set (hence measurable) if $B$ is an open bisection and $\G$ is étale.

Conversely, let $\mu$ be an invariant probability Radon measure on $\Gz$. Then it follows from \cite[Lemma 4.2]{LR19} (see also \cite[Proposition II.5.4]{Ren80}) that 
\begin{equation}\label{EQ:tau mu}
\tau_\mu: a \mapsto \int_{\Gz} E(a) \d\mu \quad \text{for} \quad a \in C^*_r(\G)
\end{equation}
is a tracial state on $C^*_r(\G)$, called the \emph{tracial state associated to $\mu$}. Similarly, we can also consider the tracial state associated to $\mu$ on $C^*_\nu(\G)$ (with the same notation)
\begin{equation}\label{EQ:tau mu max}
\tau_\mu: a \mapsto \int_{\Gz} E(q_\nu(a)) \d\mu \quad \text{for} \quad a \in C^*_\nu(\G),
\end{equation}
where $q_\nu: C^*_\nu(\G) \longrightarrow C^*_r(\G)$ is the canonical quotient map mentioned in (\ref{EQ:canonical quotient map general}).

\section{canonical tracial states and essential freeness}\label{sec:trace}

%In this section, we would like to focus on a special class of ideals coming from tracial states and apply the machinery developed in previous sections to study their properties.

%\subsection{Tracial states on groupoid $C^*$-algebras and uniqueness}\label{ssec:tracial state on algebras}

%Let $\G$ be a locally compact Hausdorff and \'{e}tale groupoid, and $\tau$ be a tracial state on $C^*_r(\G)$. Then $\tau|_{C_0(\Gz)}$ is a state on $C_0(\Gz)$, and hence corresponds to a probability Radon measure $\mu$ on $\Gz$ according to the Riesz representation theorem, called the \emph{associated measure}. It was recorded in \cite[Lemma 4.1]{LR19} that such $\mu$ is invariant. 

In this section, we study the essential freeness of \'{e}tale groupoids via their canonical tracial states. Let us start with the following definition:

\begin{defn}\label{defn:standard tracial states}
Let $\G$ be a locally compact Hausdorff and \'{e}tale groupoid, and $\|\cdot\|_\nu$ be a $C^*$-norm on $C_c(\G)$ dominating the reduced $C^*$-norm. A tracial state $\tau$ on $C^*_\nu(\G)$ is called \emph{canonical} if $\tau=\tau_\mu$, where $\mu$ is the measure associated to $\tau$ on $\Gz$. 
\end{defn}

The following is elementary but useful in what follows:

\begin{lem}\label{lem:char for standard}

A tracial state $\tau$ on $C^*_\nu(\G)$ is canonical \emph{if and only if} $\tau(f) = 0$ for all $f \in C_c(\G \setminus \Gz)$. 
\end{lem}

\begin{proof}
The forward implication is clear. Now we assume that $\tau(f) = 0$ for all $f \in C_c(\G \setminus \Gz)$. As $ C_c(\G)$ is dense in $C^*_\nu(\G)$, it suffices to show that $\tau(f) = \tau_\mu(f)$ for any $f \in C_c(\G)$, where $\mu$ is the measure associated to $\tau$. Since $\G$ is \'{e}tale, we have the decomposition $f = f|_{\Gz} + f|_{\G \setminus \Gz}$ in $C_c(\G)$. Hence by assumption, we have $\tau(f) = \tau(f|_{\Gz})$, which finishes the proof by the definition of the canonical trace (\ref{EQ:canonical trace}).
%We note that \'{e}taleness and Hausdorffness of $\G$ imply that $\Gz$ is clopen in $\G$ so that $\tau(f) = \tau_\mu(f)$ for $f \in C_c(\Gz)$ by the definition of $\mu$. Finally, we conclude that $\tau$ is canonical by the assumption.
\end{proof}

It is interesting to know when every tracial state on $C^*_\nu(\G)$ is canonical. Recall from \cite[Lemma 4.3]{LR19} (see also \cite[Proposition II.5.4]{Ren80}) that if the groupoid $\G$ is principal, then every tracial state on $C^*_r(\G)$ is canonical. We would like to weaken the condition of being principal to the following notion of essential freeness.

\begin{defn}\label{defn:ess free}
For a locally compact Hausdorff and \'{e}tale groupoid $\G$ and an invariant probability Radon measure $\mu$ on $\Gz$, we say that $\G$ is \emph{essentially free with respect to $\mu$} if for any pre-compact open bisection $B \subseteq \G \setminus \Gz$, we have $\mu(\fix(\alpha_B)) = 0$. We say that $\G$ is \emph{essentially free} if $\G$ is essentially free with respect to any invariant probability Radon measure on $\Gz$.
\end{defn}

The following proposition is perhaps known to experts at least for the maximal $C^*$-norm (see \cite[Corollary~1.2]{Nes13}). We provide here a self-contained proof, because we cannot find the explicit statement we need in the literature. 

\begin{prop}\label{prop:ess. free implies unique trace}
Let $\G$ be a locally compact Hausdorff and \'{e}tale groupoid, and $\|\cdot\|_\nu$ be a $C^*$-norm on $C_c(\G)$ dominating the reduced $C^*$-norm. If $\tau$ is a tracial state on $C^*_\nu(\G)$ with the associated measure $\mu$ on $\Gz$ such that $\G$ is essentially free with respect to $\mu$, then $\tau$ is canonical. 
\end{prop}

\begin{proof}

Our proof here is mainly inspired by the one for \cite[Proposition~1.1]{LLN09}. By Lemma \ref{lem:char for standard}, it suffices to show that $\tau(f)=0$ for any $f \in C_c(\G \setminus \Gz)$. By decomposing $f$ into its positive part and negative part, it suffices to show that $\tau(f)=0$ for any $f \in C_c(\G \setminus \Gz)$ which is point-wise non-negative. Using an argument of partitions of unity, we can additionally assume that $\supp f \subseteq B$ for some pre-compact open bisection $B$ such that $\overline{B} \subseteq B_0$ for another pre-compact open bisection $B_0\subseteq \G \setminus \Gz$. 

First we suppose that $\supp f \cap \{\gamma \in B: \s(\gamma) = \r(\gamma)\} = \emptyset$. Then for any $\gamma \in \supp f$, there exists  an open neighbourhood $W_\gamma$ of $\gamma$ such that $\s(W_\gamma) \cap \r(W_\gamma) = \emptyset$. Since $\supp f$ is compact, we can choose a finite cover $\{W_{\gamma_1}, \dotsc, W_{\gamma_N}\}$ for $\supp f$, and take a partition of unity $\{\rho_{\gamma_1}, \dotsc, \rho_{\gamma_N}\}$. Then we can write $f=\sum_{n=1}^N (\rho_n \cdot f)$, where $\rho_n \cdot f$ denotes the point-wise product of $\rho_n$ and $f$ as in Section \ref{ssec:groupoid C*-algebra}. For each $n=1,\dotsc, N$, take $h_n\in C_c(\s(W_{\gamma_n}))$ such that $h_n|_{\s(\supp (\rho_n \cdot f))} =1$ and $0 \leq h_n \leq 1$. A direct calculation as in the proof of \cite[Lemma 4.3]{LR19} shows that  that $(\rho_n \cdot f) \ast h_n = \rho_n \cdot f$ while $h_n \ast (\rho_n \cdot f)=0$. Hence, we obtain $\tau(\rho_n \cdot f) = \tau((\rho_n \cdot f) \ast h_n) = \tau(h_n \ast (\rho_n \cdot f)) = 0$, which implies that $\tau(f) = \sum_{n=1}^N \tau(\rho_n \cdot f) = 0$, as required.

%Hence it suffices to show that $\tau(\rho_n \cdot f) = 0$ for each $n$. Therefore, without loss of generality, we can further assume that $\supp f \subseteq W$ for some open $W \subseteq B$ such that $\s(W) \cap \r(W) = \emptyset$. Take $h\in C_c(\s(W))$ such that $h|_{\s(\supp f)} =1$ and $0 \leq h \leq 1$. A direct calculation as in the proof of \cite[Lemma 4.3]{LR19} shows that  that $f \ast h = f$ while $h \ast f=0$. Hence, we obtain $\tau(f) = \tau(f \ast h) = \tau(h \ast f) = 0$ as required.

Now we suppose that $\supp f \cap \{\gamma \in B: \s(\gamma) = \r(\gamma)\} \neq \emptyset$. Note that 
\[
\fix(\alpha_B) =  \{x\in \s(B): \alpha_B(x)=x\} \subseteq \{x\in \s(\overline{B}): \alpha_{\overline{B}} (x) = x\} = \fix(\alpha_{\overline{B}}),
\]
and $\fix(\alpha_{\overline{B}})$ is closed in $\s(\overline{B})$. It follows that $\overline{\fix(\alpha_B)}$ is compact and is contained in $\fix(\alpha_{B_0})$. By the essential freeness of $\mu$, we know that $\mu(\fix(\alpha_{B_0})) = 0$ and hence $\mu(\overline{\fix(\alpha_B)}) = 0$ as well. 

Since $\mu$ is a Radon measure, it is outer regular. Hence given $\varepsilon>0$, we can take an open set $U\subseteq \Gz$ containing the closure of $\s(\supp f) \cap \fix(\alpha_B)$ and an open set $V \supseteq \overline{U}$ such that $\mu(V) < \varepsilon$. Since $\s(\supp f)$ is compact and $V$ is open, $\s(\supp f) \setminus V$ is also compact. Hence we can take $\rho_\varepsilon \in C_c(\Gz)$ such that $0\leq \rho_\varepsilon\leq 1$, $\rho_\varepsilon|_{\s(\supp f) \setminus V} \equiv 1$ and $\rho_\varepsilon|_{\overline{U}} \equiv 0$. 

We aim to apply the argument of the second paragraph above to the function $f \ast \rho_\varepsilon$, which has support in $B$ by Lemma~\ref{lem:easy calculation 0}. Hence we must first show that $\supp(f \ast \rho_\varepsilon) \cap \{\gamma \in B: \s(\gamma) = \r(\gamma)\} = \emptyset$. Note that
%Let us consider $f \ast \rho_\varepsilon$, which has support in $B$ by Lemma~\ref{lem:easy calculation 0}. Then 
\begin{align*}
\s(\supp(f \ast \rho_\varepsilon) \cap \{\gamma \in B: \s(\gamma) = \r(\gamma)\}) & \subseteq \s(\supp f \cap \{\gamma \in B: \s(\gamma) = \r(\gamma)\})\\
& = \s(\supp f) \cap \fix(\alpha_B) \subseteq U.
\end{align*}
Assume that there exists $\gamma \in \supp(f \ast \rho_\varepsilon) \cap \{\gamma \in B: \s(\gamma) = \r(\gamma)\}$. Since $\gamma \in (\s|_B)^{-1}(U) \cap \supp(f \ast \rho_\varepsilon)$, we can choose a net $\{\gamma_\lambda\}_\lambda$ in $(\s|_B)^{-1}(U)$ converging to $\gamma$ such that $(f\ast \rho_\varepsilon)(\gamma_\lambda) \neq 0$, then Lemma \ref{lem:easy calculation 0} implies that $\rho_\varepsilon(\s(\gamma_\lambda)) \neq 0$ for each $\lambda$. We reach a contradiction since $\s(\gamma_\lambda) \in U$ and $\rho_\varepsilon|_{U} \equiv 0$. Therefore, the same analysis in the second paragraph of this proof shows that $\tau(f \ast \rho_\varepsilon) = 0$.

Take $\eta \in C_c(B_0)$ such that $\eta|_{\overline{B}} \equiv 1$ and $||\eta ||_{\infty}=1$, and define $f_0 \in C_c(\Gz)$ by $f_0(x)=f \circ (\s|_{B_0})^{-1}(x)$ if $x\in \s(B_0)$ and zero otherwise. As $\supp f \subseteq B_0$, $f_0$ is a continuous function on $\Gz$ such that $\supp f_0= \s(\supp f)$. Then Lemma \ref{lem:easy calculation 0} implies that
\[
(\eta \ast f_0)(\gamma) = \eta(\gamma) \cdot f_0(\s(\gamma)) = \eta(\gamma) \cdot f(\gamma) = f(\gamma), \quad \forall \gamma \in \G, 
\]
which means that $\eta \ast f_0 = f$. Afterwards we obtain
\begin{equation}\label{EQ:cal for tau 1}
|\tau(f)| = |\tau(f) - \tau(f \ast \rho_\varepsilon)| = |\tau(\eta \ast (f_0-f_0 \cdot \rho_\varepsilon))|.
\end{equation}
Since $f_0-f_0 \cdot \rho_\varepsilon$ is point-wise non-negative, we write
\[
\eta \ast (f_0-f_0 \cdot \rho_\varepsilon) = \left( \eta \ast (f_0-f_0 \cdot \rho_\varepsilon)^{1/2} \right) \ast (f_0-f_0 \cdot \rho_\varepsilon)^{1/2}.
\]
Using the Cauchy--Schwarz inequality $|\tau(y^*x)|^2 \leq \tau(x^* x) \cdot \tau(y^*y)$ for  $x=(f_0-f_0 \cdot \rho_\varepsilon)^{1/2}$ and $y^*= \eta \ast (f_0-f_0 \cdot \rho_\varepsilon)^{1/2}$, we obtain
\begin{equation}\label{EQ:cal for tau 2}
|\tau(\eta \ast (f_0-f_0 \cdot \rho_\varepsilon))|^2 \leq \tau(f_0 - f_0 \cdot \rho_\varepsilon) \cdot \tau(\eta \ast (f_0 - f_0 \cdot \rho_\varepsilon) \ast \eta^*).
\end{equation}
Using properties of tracial states, we have
\begin{equation}\label{EQ:cal for tau 3}
\tau(\eta \ast (f_0 - f_0 \cdot \rho_\varepsilon) \ast \eta^*) = \tau((f_0 - f_0 \cdot \rho_\varepsilon)^{1/2} \ast \eta^* \ast \eta \ast (f_0 - f_0 \cdot \rho_\varepsilon)^{1/2}) \leq \tau(f_0 - f_0 \cdot \rho_\varepsilon),
\end{equation}
where we use $\|\eta\|_\infty =1$ for the second inequality.
Combining (\ref{EQ:cal for tau 1}), (\ref{EQ:cal for tau 2}) and (\ref{EQ:cal for tau 3}), we obtain
\[
|\tau(f)| \leq \tau(f_0 - f_0 \cdot \rho_\varepsilon).
\]
%\begin{align*}
%|\tau(f)| &= |\tau(f) - \tau(f \ast \rho_\varepsilon)|\\& = |\tau(\eta \ast (f_0-f_0 \cdot \rho_\varepsilon))| \\
%&\leq \tau(|f_0 - f_0 \cdot \rho_\varepsilon|)^{1/2} \cdot \tau(\eta \ast |f_0 - f_0 \cdot \rho_\varepsilon| \ast \eta^*)^{1/2}\\
%& \leq \tau(|f_0 - f_0 \cdot \rho_\varepsilon|)\cdot ||\eta^* \ast \eta||_\infty\\
%&=\tau(|f_0 - f_0 \cdot \rho_\varepsilon|),
%\end{align*}
%where we use the Cauchy--Schwarz inequality $|\tau(y^*x)|^2 \leq \tau(x^* x) \cdot \tau(y^*y)$ for any $x,y\in C^*_\nu(\G)$ in the first inequality, and the tracial property of $\tau$ in the second inequality. 
Finally, from $\mu(V)<\epsilon$ we have that
\[
\tau(f_0 - f_0 \cdot \rho_\varepsilon) = \int_{\Gz} f_0 \cdot (1-\rho_\varepsilon) \d \mu = \int_{\s(\supp f) \cap V} f_0 \cdot (1-\rho_\varepsilon) \d \mu \leq \varepsilon \cdot \|f_0\|_\infty,
\]
which goes to $0$ as $\varepsilon \to 0$. Therefore, we conclude that $\tau(f)=0$, as desired.
\end{proof}

\begin{rem}\label{rem:ess. free}
If $\iso(\G):=\{\gamma \in \G: \r(\gamma)  = \s(\gamma)\}$ denotes the isotropy groupoid of an étale groupoid $\G$, then we have
\[
\s(\iso(\G) \setminus \Gz) = \bigcup\{\fix(\alpha_B): B \text{ is a pre-compact open bisection in } \G \setminus \Gz\}.
\] 
If $\G$ is $\sigma$-compact, then $\s(\iso(\G) \setminus \Gz)$ is a countable union of measurable sets. In particular, it is measurable. Therefore, $\G$ is in this case essentially free with respect to $\mu$ if and only if $\mu(\s(\iso(\G) \setminus \Gz))=0$.

\end{rem}

%We introduce the following:
%
%\begin{defn}\label{defn:ess free}
%For a locally compact Hausdorff and \'{e}tale groupoid $\G$ and an invariant probability Radon measure $\mu$ on $\Gz$, we say that $\G$ is \emph{essentially free with respect to $\mu$} if for any pre-compact open bisection $B \subseteq \G \setminus \Gz$, we have $\mu(\fix(\alpha_B)) = 0$. We say that $\G$ is \emph{essentially free} if $\G$ is essentially free with respect to any invariant probability Radon measure on $\Gz$.
%\end{defn}

If we assume that $\Gz$ is compact then the convex set $M(\G)$ of invariant probability Radon measures on $\Gz$ and the tracial state space $T(C^*_\nu(\G))$ of $C^*_\nu(\G)$ are both compact in the weak$^\ast$-topology. Then the following corollary of Proposition~\ref{prop:ess. free implies unique trace} has generalised \cite[Proposition 3.1]{ABBL20}, because almost finite ample groupoids with compact unit space are always essentially free by \cite[Remark~6.6]{MR2876963}:

\begin{cor}\label{cor:generalisation for ABBL20, Prop2.2}
Let $\G$ be a locally compact Hausdorff and \'{e}tale groupoid with compact unit space which is also essentially free. Then the canonical map $\tau \mapsto \mu_\tau$ from $T(C^*_\nu(\G))$ to $M(\G)$ is an affine homeomorphism, and hence we can identify their extreme boundaries $\partial_e T(C^*_\nu(\G)) = \partial_e M(\G)$. In particular, this holds for both maximal and reduced $C^*$-norms.
\end{cor}

In the following, we would like to study Question \ref{Ques:converse} for a locally compact Hausdorff and \'{e}tale groupoid such that $\Gz$ is not necessarily compact. For that we consider an auxiliary trace $\tau^\fix_\mu: C_c(\G) \to \CC$ associated to a given invariant probability Radon measure $\mu$ on $\Gz$. The key point is that $\tau^\fix_\mu$ reveals the complete information of essential freeness with respect to $\mu$ (see Lemma~\ref{lem:tau fix char strong eff.} for details). Since its construction is a bit complicated, we divide it into several steps.

Firstly, for any pre-compact open bisection $B$ and $g\in C_c(B)$ we define $g_B\in C_c(\s(B))$ by $g_B(x) := g((\s|_B)^{-1}(x))$ for $x\in \s(B)$. Similarly, we define $g^B \in C_c(\r(B))$ by $g^B(x) := g((\r|_B)^{-1}(x))$ for $x\in \r(B)$. For such $g$ and $B$, we define 
\begin{equation}\label{EQ:fix trace bisection case}
\tau^\fix_\mu(g) := \int_{\fix(\alpha_B)} g_B \d \mu.
\end{equation}
The following observation shows that $\tau^\fix_\mu(g)$ is well-defined.

\begin{lem}\label{lem:well defineness for tau fix on B1}
Assume that $B_1, B_2$ are pre-compact open bisections and $g\in C_c(B_1) \cap C_c(B_2)$. Then we have
\[
\int_{\fix(\alpha_{B_1})} g_{B_1} \d \mu = \int_{\fix(\alpha_{B_2})} g_{B_2} \d \mu.
\]
\end{lem}

\begin{proof}
Taking $B = B_1 \cap B_2$, it suffices to show that $\int_{\fix(\alpha_{B_i})} g_{B_i} \d \mu = \int_{\fix(\alpha_{B})} g_{B} \d \mu$ for $i=1,2$. Note that $\supp (g) \subseteq B_1 \cap B_2 = B$, and hence $\supp(g_B) \subseteq \s(B)$ and $g_B = g_{B_i}$ for $i=1,2$. Therefore, for $i=1,2$ we have
\[
\int_{\fix(\alpha_{B_i})} g_{B_i} \d \mu = \int_{\fix(\alpha_{B_i}) \cap \s(B)} g_B \d \mu.
\]
We note that $x\in \fix(\alpha_{B_i}) \cap \s(B)$ if and only if there exists $\gamma_i \in B_i$ such that $\s(\gamma_i) = \r(\gamma_i) = x \in \s(B)$, which implies that $\gamma_i \in B$ and hence $x\in \fix(\alpha_B)$. It follows that $\fix(\alpha_{B_i}) \cap \s(B) = \fix(\alpha_B)$ for $i=1,2$, as desired.
\end{proof}

Moreover, we have the following:

\begin{lem}\label{lem:well defineness for tau fix on B}
Let $g \in C_c(\G)$ with $g= \sum_{i=1}^n g_i = \sum_{j=1}^m h_j$, where $g_i \in C_c(B_i)$ for some pre-compact open bisection $B_i$ and $h_j \in C_c(D_j)$ for some pre-compact open bisection $D_j$. Then we have 
\[
\sum_{i=1}^n \tau^\fix_\mu(g_i) = \sum_{j=1}^m \tau^\fix_\mu(h_j).
\]
\end{lem}

\begin{proof}
Since $K:=\bigcup_{i=1}^n \supp(g_i) \cup \bigcup_{j=1}^m \supp(h_j)$ is compact, we can take a finite open cover $\U=\{U_k:k=1,\dotsc,N\}$ of $K$ such that each $U_k$ is a pre-compact open bisection. Then we take a partition of unity $\{\rho_k:k=1,\dotsc, N \}$ subordinate to $\U$ such that $\sum_{k=1}^{N} \rho_k \equiv 1$ on $K$. In particular, we have $g_i = \sum_{k=1}^{N} (\rho_k \cdot g_i)$ for any $i=1,\dotsc, n$, where $\rho_k \cdot g_i$ means the point-wise product as in Section \ref{ssec:groupoid C*-algebra}. As both $\supp(g_i)$ and $\supp(\rho_k \cdot g_i)$ are contained in $B_i$, it follows from Lemma \ref{lem:well defineness for tau fix on B1} that
\begin{equation}\label{EQ:well define}
\tau^\fix_\mu(g_i) = \int_{\fix(\alpha_{B_i})} (g_i)_{B_i} \d \mu \quad \text{and} \quad \tau^\fix_\mu(\rho_k \cdot g_i) = \int_{\fix(\alpha_{B_i})} (\rho_k \cdot g_i)_{B_i} \d \mu.
\end{equation}
For each $x\in \s(B_i)$, we have $(\rho_k \cdot g_i)_{B_i}(x) = \rho_k((s|_{B_i})^{-1}(x)) \cdot g_i((s|_{B_i})^{-1}(x))$. Hence, we obtain $\sum_{k=1}^{N} (\rho_k \cdot g_i)_{B_i} = (g_i)_{B_i}$, which together with (\ref{EQ:well define}) implies that
\[
\tau^\fix_\mu(g_i) = \int_{\fix(\alpha_{B_i})} (g_i)_{B_i}  \d \mu= \sum_{k=1}^{N} \int_{\fix(\alpha_{B_i})}(\rho_k \cdot g_i)_{B_i}\d \mu = \sum_{k=1}^{N}\tau^\fix_\mu(\rho_k \cdot g_i).
\]
So we obtain
\begin{equation}\label{EQ:well define 1}
\sum_{i=1}^n \tau^\fix_\mu(g_i) = \sum_{i=1}^n \sum_{k=1}^{N} \tau^\fix_\mu(\rho_k \cdot g_i) = \sum_{k=1}^{N} \sum_{i=1}^n \tau^\fix_\mu(\rho_k \cdot g_i).
\end{equation}
Similarly, we also have
\begin{equation}\label{EQ:well define 2}
\sum_{j=1}^m \tau^\fix_\mu(h_j) = \sum_{k=1}^{N} \sum_{j=1}^m \tau^\fix_\mu(\rho_k \cdot h_j).
\end{equation}
For a fixed $k\in \{1, \dotsc, N\}$, using a similar argument as above we have
\begin{equation}\label{EQ:well define 3}
\sum_{i=1}^n \tau^\fix_\mu(\rho_k \cdot g_i)  = \tau^\fix_\mu \big(\sum_{i=1}^n \rho_k \cdot g_i\big) = \tau^\fix_\mu(\rho_k \cdot g) = \tau^\fix_\mu\big(\sum_{j=1}^m \rho_k \cdot h_j\big) = \sum_{j=1}^m \tau^\fix_\mu(\rho_k \cdot h_j).
\end{equation}
Finally, we conclude the proof by  (\ref{EQ:well define 1}), (\ref{EQ:well define 2}) and (\ref{EQ:well define 3}).
\end{proof}

Lemma \ref{lem:well defineness for tau fix on B} allows us to give the following definition of $\tau^\fix_\mu$ on $C_c(\G)$.

\begin{defn}\label{defn:fix trace}
Let $\G$ be a locally compact Hausdorff and \'{e}tale groupoid, and $\mu$ be an invariant probability Radon measure on $\Gz$. The associated \emph{fixed point trace} is the linear map $\tau^\fix_\mu: C_c(\G) \to \CC$ defined as follows: Suppose that $g\in C_c(\G)$ with $g= \sum_{i=1}^n g_i$ such that each $g_i \in C_c(B_i)$ for some pre-compact open bisection $B_i$. Then we define 
\begin{equation}\label{EQ:tau fix general}
\tau^\fix_\mu(g) :=\sum_{i=1}^n \tau^\fix_\mu(g_i),
\end{equation}
where $\tau^\fix_\mu(g_i)$ is given in (\ref{EQ:fix trace bisection case}).
\end{defn}

We record the following fact, which is straightforward from the construction.

\begin{lem}\label{lem:tau fix defined}
For $f\in C_c(\Gz)$, we have $\tau^\fix_\mu(f) = \int_{\Gz} f \d \mu$.
\end{lem}

The next lemma shows that $\tau^\fix_\mu$ is positive and has the tracial property.

\begin{lem}\label{lem:tau fix is a trace}
For any $a \in C_c(\G)$, we have $\tau^\fix_\mu(a^*a) = \tau^\fix_\mu(aa^*) \geq 0$.
\end{lem}

\begin{proof}
Assume that $a= \sum_{i=1}^n g_i$ such that each $g_i \in C_c(B_i)$ for some pre-compact open bisection $B_i$. Then we have $a^*a = \sum_{i,j=1}^n g_i^* \ast g_j$ and $aa^* = \sum_{i,j=1}^n g_j \ast g_i^*$. Note that for any $\gamma \in \G$ and $i,j=1, \dotsc, n$, we have
\[
(g_j \ast g_i^*)(\gamma) = \sum_{\alpha \in \G_{\s(\gamma)}} g_j(\gamma \alpha^{-1}) \cdot g_i^*(\alpha) = \sum_{\alpha \in \G_{\s(\gamma)}} g_j(\gamma \alpha^{-1}) \cdot\overline{g_i(\alpha^{-1})}.
\] 
If $(g_j \ast g_i^*)(\gamma) \neq 0$, then there exists a unique $\alpha^{-1} \in B_i$ such that $\r(\alpha^{-1}) = \s(\gamma)$ and $(g_j \ast g_i^*)(\gamma) = g_j(\gamma \alpha^{-1}) \cdot\overline{g_i(\alpha^{-1})} \neq 0$. This shows that 
\[
\supp(g_j \ast g_i^*) \subseteq \supp(g_j) \cdot (\supp(g_i))^{-1} \subseteq B_j \cdot B_i^{-1}.
\]
Since multiplication and inversion are continuous and $\G$ is Hausdorff, the set $B_j \cdot B_i^{-1}$ is precompact, and it follows from Lemma \ref{lem:multiplication of bisection} that $B_j \cdot B_i^{-1}$ is also an open bisection. Moreover, for any $\gamma \in B_j \cdot B_i^{-1}$ there exist unique $\beta \in B_j$ and $\alpha \in B_i^{-1}$ such that $\gamma = \beta\alpha$, and hence $(g_j \ast g_i^*)(\gamma) =g_j(\beta) \cdot \overline{g_i(\alpha^{-1})}$.

It is also worth noticing that 
\[
\s(B_j \cdot B_i^{-1}) = \r\big((\s|_{B_i})^{-1}(\s(B_j) \cap \s(B_i))\big) \quad \text{and} \quad \r(B_j \cdot B_i^{-1}) = \r\big((\s|_{B_j})^{-1}(\s(B_j) \cap \s(B_i))\big).
\]
Thus by definition, a direct calculation shows that $\alpha_{B_j \cdot B_i^{-1}} = \alpha_{B_j} \circ \alpha_{B_i}^{-1}$. Moreover, if $x\in \fix(\alpha_{B_j} \circ \alpha_{B_i}^{-1})$ then we take $\gamma \in B_j \cdot B_i^{-1}$ such that $\s(\gamma) =x$. Writing $\gamma = \beta \alpha^{-1}$ for $\beta \in B_j$ and $\alpha \in B_i$, then $\s(\alpha) = \s(\beta)$ and $\r(\alpha) = x = \alpha_{B_j} \circ \alpha_{B_i}^{-1} (x) =\alpha_{B_j}(\s(\alpha))= \alpha_{B_j}(\s(\beta))= \r(\beta)=\r(\gamma)$. So for $x\in \fix(\alpha_{B_j} \circ \alpha_{B_i}^{-1})$ we obtain
\[
(g_j \ast g_i^*)_{B_j \cdot B_i^{-1}}(x) = (g_j)^{B_j}(x) \cdot \overline{(g_i)^{B_i}(x)}.
\]
Therefore, we obtain
\begin{align}\label{EQ:trace check 2}
\tau^\fix_\mu(aa^*) &= \sum_{i,j=1}^n \tau^\fix_\mu(g_j \ast g_i^*) = \sum_{i,j=1}^n \int_{\fix(\alpha_{B_j \cdot B_i^{-1}})} (g_j \ast g_i^*)_{B_j \cdot B_i^{-1}}(x) \d \mu(x)  \nonumber \\ 
&= \sum_{i,j=1}^n \int_{\fix(\alpha_{B_j} \circ \alpha_{B_i}^{-1})} (g_j)^{B_j}(x) \cdot \overline{(g_i)^{B_i}(x)}\d \mu(x).
\end{align}

On the other hand, we have that
\begin{align}\label{EQ:trace check 3}
\tau^\fix_\mu(a^*a) &= \sum_{i,j=1}^n \tau^\fix_\mu(g_i^* \ast g_j) = \sum_{i,j=1}^n \int_{\fix(\alpha_{B_i^{-1} \cdot B_j})} (g_i^* \ast g_j)_{B_i^{-1} \cdot B_j}(x) \d \mu(x) \nonumber \\ 
&= \sum_{i,j=1}^n \int_{\alpha_{B_i} \fix(\alpha_{B_i}^{-1} \circ \alpha_{B_j})} (g_i^* \ast g_j)_{B_i^{-1} \cdot B_j} \circ \alpha_{B_i}^{-1}(x) \d \mu(x),
\end{align}
where the last equality holds since $\mu$ is invariant. For $x\in \alpha_{B_i} \fix(\alpha_{B_i}^{-1} \circ \alpha_{B_j})$, we have $\alpha_{B_i}^{-1} \circ \alpha_{B_j} \circ \alpha_{B_i}^{-1}(x) = \alpha_{B_i}^{-1}(x)$, which implies that $x= \alpha_{B_j} \circ \alpha_{B_i}^{-1}(x)$. Hence, we conclude that $\alpha_{B_i} \fix(\alpha_{B_i}^{-1} \circ \alpha_{B_j})=\fix(\alpha_{B_j} \circ \alpha_{B_i}^{-1})$. Let $y: = \alpha_{B_i}^{-1}(x)$, then we have $y\in \s(B_j) \cap \s(B_i)$. Take the unique $\beta' \in B_j$ with $\s(\beta') = y$, and the unique $\alpha' \in B_i$ with $\s(\alpha') = y$. Hence, we have $\r(\alpha') = x$. As $x= \alpha_{B_j} (y)=\alpha_{B_j} (\s(\beta'))=\r(\beta')$, we obtain 
\[
(g_i^* \ast g_j)_{B_i^{-1} \cdot B_j} \circ \alpha_{B_i}^{-1} (x) = (g_i^* \ast g_j)_{B_i^{-1} \cdot B_j}(y) = \overline{g_i(\alpha')} \cdot g_j(\beta') = \overline{(g_i)^{B_i}(x)} \cdot (g_j)^{B_j}(x).
\]
From (\ref{EQ:trace check 3}) we obtain
\begin{align}\label{EQ:trace check 4}
\tau^\fix_\mu(a^*a) 
&= \sum_{i,j=1}^n \int_{\fix(\alpha_{B_j} \circ \alpha_{B_i}^{-1})} \overline{(g_i)^{B_i}(x)} \cdot (g_j)^{B_j}(x) \d \mu(x). 
\end{align}
By (\ref{EQ:trace check 2}) and (\ref{EQ:trace check 4}) we conclude $\tau^\fix_\mu(aa^*) = \tau^\fix_\mu(a^*a)$, as desired.

Now we aim to show that $\tau^\fix_\mu(aa^*) \geq 0$. By (\ref{EQ:trace check 2}) we have
\begin{align*}
\tau^\fix_\mu(aa^*) &= \sum_{i,j=1}^n \int_{\fix(\alpha_{B_j} \circ \alpha_{B_i}^{-1})} (g_j)^{B_j}(x) \cdot \overline{(g_i)^{B_i}(x)}\d \mu(x)\\
 &=  \int_{\Gz} \sum_{i,j=1}^n \chi_{B_{ij}}(x) \cdot (g_j)^{B_j}(x) \cdot \overline{(g_i)^{B_i}(x)}\d \mu(x),
\end{align*}
where $B_{ij}:=\{x \in \r(B_i) \cap \r(B_j): \alpha_{B_i}^{-1}(x) = \alpha_{B_j}^{-1}(x)\}$. 
Let us consider the function 
\[
F(x):= \sum_{i,j=1}^n \chi_{B_{ij}}(x) \cdot (g_j)^{B_j}(x) \cdot \overline{(g_i)^{B_i}(x)} \quad \text{for} \quad  x\in \Gz. 
\]
It suffices to show that $F(x) \geq 0$ for each $x\in \Gz$. Fix an $x\in \Gz$, and let $I:=\{i\in \{1,\dotsc,n\}: x \in \r(B_i)\}$.
Note that there is an equivalence relation on elements of $I$ given by $i \sim j$ if and only if $x \in B_{ij}$. Hence we can decompose $I=I_1 \sqcup I_2 \sqcup \dotsb \sqcup I_N$ such that for any $i,j \in I_k$ (where $k=1, \dotsc, N$) we have $\alpha_{B_i}^{-1}(x) = \alpha_{B_j}^{-1}(x)$, and for any $i \in I_k$ and $j \in I_l$ where $k \neq l$ and $k,l=1, \dotsc, N$ we have $\alpha_{B_i}^{-1}(x) \neq \alpha_{B_j}^{-1}(x)$. Hence, we have
\[
F(x) = \sum_{k=1}^N \sum_{i,j \in I_k} (g_j)^{B_j}(x) \cdot \overline{(g_i)^{B_i}(x)} = \sum_{k=1}^N \big| \sum_{i \in I_k} (g_i)^{B_i}(x) \big|^2 \geq 0,
\]
where we use the elementary calculation $|\sum_{i=1}^n \lambda_i|^2 = \sum_{i,j=1}^n \lambda_i\overline{\lambda_j}$ for $\lambda_1,\ldots,\lambda_n \in \CC$ for the second equality. 
So we have completed the proof.
\end{proof}

We also need the following property of $\tau^\fix_\mu$:

\begin{lem}\label{lem:tau fix app unit}
Let $\{u_i\}_{i \in I} \subseteq C_c(\Gz)$ be any approximate unit for $C_0(\Gz)$. Then we have $\lim_{i \in I} \tau_{\mu}^{\fix}(u_i) = 1$.
\end{lem}

\begin{proof}
By Lemma \ref{lem:tau fix defined}, $\tau_{\mu}^{\fix}|_{C_c(\Gz)}$ can be extended to a tracial state on $C_0(\Gz)$ (with the same notation). Hence, we have $\lim_{i \in I} \tau^{\fix}_\mu(u_i) = \|\tau_{\mu}^{\fix}|_{C_0(\Gz)}\| = 1$ (see, \emph{e.g.}, \cite[Theorem 3.3.3]{Mur90}).
\end{proof}

Using the GNS construction together with Lemma \ref{lem:tau fix app unit}, we obtain the following:

\begin{cor}\label{cor:tau fix extended to Cmax}
The map $\tau^\fix_\mu: C_c(\G) \to \CC$ defined in Definition~\ref{defn:fix trace} can be extended to a tracial state on $C^*(\G)$, still denoted by $\tau^\fix_\mu$.
\end{cor}

\begin{proof}
Define $(\cdot, \cdot)$ on $C_c(\G)$ by $(a,b):=\tau^\fix_\mu(b^*a)$, and set $N:=\{a \in C_c(\G): (a,a) = 0\}$. Set $\langle \cdot, \cdot \rangle$ on $C_c(\G) / N$ by $\langle [a], [b] \rangle := (a,b)$ for $a,b \in C_c(\G)$, which is well-defined. Then let $\H$ be the completion of $C_c(\G) / N$ with respect to $\langle \cdot, \cdot \rangle$, which is a Hilbert space. Also, define the $\ast$-representation $\pi: C_c(\G) \to \B(\H)$ by $\pi(a)([b]):=[ab]$, where $a,b \in C_c(\G)$. By definition, $\pi$ can be extended to a $\ast$-representation (with the same notation) $\pi: C^*(\G) \to \B(\H)$. For each compact $K \subseteq \Gz$, choose $\rho_K \in C_c(\Gz)$ with range in $[0,1]$ and $\rho_K|_K \equiv 1$. Then $\{\rho_K: K \text{ is a  compact subset of }\Gz\}$ forms an approximate unit for $C_0(\Gz)$. 

We claim that $\{[\rho_K^{1/2}]\}_{K}$ is a Cauchy net in $\H$. In fact, we have
\begin{align*}
\left\|[\rho_K^{1/2}] - [\rho_L^{1/2}]\right\|^2_\H &= \tau^\fix_\mu\left((\rho_K^{1/2} - \rho_L^{1/2})(\rho_K^{1/2} - \rho_L^{1/2})\right) \\
&= \tau^\fix_\mu(\rho_K) + \tau^\fix_\mu(\rho_L) - 2\tau^\fix_\mu(\rho_K^{1/2}\rho_L^{1/2}).
\end{align*}
Taking limits for $K$ and $L$, Lemma \ref{lem:tau fix app unit} implies that $\tau^\fix_\mu(\rho_K) \to 1$ and $\tau^\fix_\mu(\rho_L) \to 1$. Moreover, Lemma \ref{lem:tau fix defined} shows that
\[
\tau^\fix_\mu(\rho_K^{1/2}\rho_L^{1/2}) = \int_{\Gz} \rho_K^{1/2}\rho_L^{1/2} \d \mu \to 1
\]
since $\mu$ is a probability Radon measure. 
This concludes the proof that $\{[\rho_K^{1/2}]\}_{K}$ is a Cauchy net in $\H$, and hence converges to a unit vector $\xi \in \H$. For $a \in C^*(\G)$, we define $\tau^{\fix}_\mu(a) :=\langle \pi(a) \xi , \xi \rangle$. It is routine to check that $\tau^\fix_\mu$ is a tracial state on $C^*(\G)$ which extends $\tau^\fix_\mu: C_c(\G) \to \CC$.
\end{proof}

The following lemma is a crucial step towards an answer to Question \ref{Ques:converse}:

\begin{lem}\label{lem:tau fix char strong eff.}
Let $\G$ be a locally compact Hausdorff and \'{e}tale groupoid, and $\mu$ be an invariant probability Radon measure on $\Gz$. Then the following are equivalent:
\begin{enumerate}
 \item $\G$ is essentially free with respect to $\mu$;
 \item $\tau^\fix_\mu(a) = \tau_\mu(a)$ for any $a\in C_c(\G)$.
\end{enumerate}
\end{lem}

\begin{proof}
``(1) $\Rightarrow$ (2)'' follows directly from definitions. For ``(2) $\Rightarrow$ (1)'': If not, then there exists a pre-compact open bisection $B \subseteq \G \setminus \Gz$ such that $\mu(\fix(\alpha_B)) \neq 0$. Then for any $g \in C_c(B)$ with $g>0$, we have $\tau^\fix_\mu(g) = \int_{\fix(\alpha_B)} g_B \d\mu >0$, while $\tau_\mu(g) = \int_{\Gz} E(g) \d \mu = 0$. This leads to a contradiction.
\end{proof}

By Proposition \ref{prop:ess. free implies unique trace} and Lemma \ref{lem:tau fix char strong eff.} we obtain the following:

\begin{cor}\label{cor:criterion for standard}
Let $\G$ be a locally compact Hausdorff and \'{e}tale groupoid with a $C^*$-norm $\|\cdot\|_\nu$ on $C_c(\G)$ dominating the reduced $C^*$-norm, and $\mu$ be an invariant probability Radon measure on $\Gz$. Assume that $\tau_\mu = \tau^\fix_\mu$ on $C_c(\G)$. Then any tracial state on $C^*_\nu(\G)$ with the associated measure being $\mu$ is canonical. 
\end{cor}

Finally, we obtain the following answer to Question \ref{Ques:converse}:

\begin{thm}\label{thm:char for ess free max case}
Let $\G$ be a locally compact Hausdorff and \'{e}tale groupoid, and $\mu$ be an invariant probability Radon measure on $\Gz$. Then the following are equivalent:
\begin{enumerate}
 \item $\G$ is essentially free with respect to $\mu$;
 \item For any $C^*$-norm $\|\cdot\|_\nu$ on $C_c(\G)$ dominating the reduced $C^*$-norm, any tracial state $\tau$ on $C^*_\nu(\G)$ with the associated measure being $\mu$ is canonical and  $\tau^\fix_\mu$ can be extended to a tracial state on $C^*_\nu(\G)$; 
 \item There exists a $C^*$-norm $\|\cdot\|_\nu$ on $C_c(\G)$ dominating the reduced $C^*$-norm such that any tracial state $\tau$ on $C^*_\nu(\G)$ with the associated measure being $\mu$ is canonical and $\tau^\fix_\mu$ can be extended to a tracial state on $C^*_\nu(\G)$; 
 \item Any tracial state $\tau$ on $C^*(\G)$ with the associated measure being $\mu$ is canonical.
\end{enumerate}
\end{thm}

\begin{proof}
``(1) $\Rightarrow$ (2)'': The first half of the sentence follows from Proposition \ref{prop:ess. free implies unique trace}. For the second statement, Lemma \ref{lem:tau fix char strong eff.} shows that $\tau^\fix_\mu(a) = \tau_\mu(a)$ for $a \in C_c(\G)$. Since $\tau_\mu$ can be extended to a tracial state on $C^*_\nu(\G)$, we conclude (2).

``(2) $\Rightarrow$ (3)'' and ``(2) $\Rightarrow$ (4)'' are trivial.

``(3) $\Rightarrow$ (1)'': By (3), $\tau^\fix_\mu$ can be extended to a tracial state on $C^*_\nu(\G)$. We note that the associated measure of $\tau^\fix_\mu$ is $\mu$ by Lemma \ref{lem:tau fix defined}. It follows from (3) again that $\tau^\fix_\mu = \tau_\mu$. Finally, we conclude (1) by Lemma \ref{lem:tau fix char strong eff.}.

``(4) $\Rightarrow$ (1)'': By Corollary \ref{cor:tau fix extended to Cmax}, $\tau^\fix_\mu$ can be extended to a tracial state on $C^*(\G)$. We note that the associated measure of $\tau^\fix_\mu$ is $\mu$ by Lemma \ref{lem:tau fix defined}. From the assumption (4), we obtain $\tau^\fix_\mu = \tau_\mu$. Finally, we conclude (1) by Lemma \ref{lem:tau fix char strong eff.}.
\end{proof}

%Now we turn to the reduced case, where the situation becomes a bit more complicated.
%\begin{thm}\label{thm:char for ess free}
%Let $\G$ be a locally compact Hausdorff and \'{e}tale groupoid, and $\mu$ be an invariant probability Radon measure on $\Gz$. 
%Then the following are equivalent:
%\begin{enumerate}
% \item For any pre-compact open bisection $B \subseteq \G \setminus \Gz$, we have $\mu(\fix(\alpha_B)) = 0$;
% \item For any tracial state $\tau$ on $C^*_r(\G)$ with the associated measure being $\mu$, we have $\tau = \tau_\mu$ and moreover, $\tau^\fix$ can be extended to a tracial state on $C^*_r(\G)$. 
%\end{enumerate}
%\end{thm}
%
%\begin{proof}
%``(1) $\Rightarrow$ (2)'': The first half sentence follows from Proposition \ref{prop:ess. free implies unique trace}. For the rest, Lemma \ref{lem:tau fix char strong eff.} shows that $\tau^\fix(a) = \tau_\mu(a)$ for $a \in C_c(\G)$. Since $\tau_\mu$ can be extended to a tracial state on $C^*_r(\G)$, we conclude (2).
%
%``(2) $\Rightarrow$ (1)'': By condition (2), $\tau^\fix$ can be extended to a tracial state on $C^*_r(\G)$. Note that the associated measure of $\tau^\fix$ is $\mu$ by Lemma \ref{lem:tau fix defined}, and hence by (2) again, we obtain $\tau^\fix = \tau_\mu$. Finally by Lemma \ref{lem:tau fix char strong eff.}, we conclude (1).
%\end{proof}

Recall that a groupoid $\G$ has the \emph{weak containment property} if $C^*(\G) \cong C^*_r(\G)$ canonically\footnote{Amenable groupoids have the weak containment property, but there are also non-amenable groupoids with the weak containment property \cite{MR3549528}.}. In this case, we obtain the following directly from Corollary \ref{cor:tau fix extended to Cmax} and Theorem \ref{thm:char for ess free max case}:

\begin{cor}\label{cor:WCP to ess free}
Let $\G$ be a locally compact Hausdorff and \'{e}tale groupoid with the weak containment property, and $\mu$ be an invariant probability Radon measure on $\Gz$. Then the following are equivalent:
\begin{enumerate}
 \item $\G$ is essentially free with respect to $\mu$;
 \item Any tracial state $\tau$ on $C^*_r(\G)$ with the associated measure being $\mu$ is canonical.
\end{enumerate}
\end{cor}

%\begin{rem}
%Note that condition (1) in Corollary \ref{cor:WCP to ess free} does not imply the weak containment property in general. For example, consider the free action of a group $\Gamma$ on its Stone-\v{C}ech boundary $\partial_\beta \Gamma$ induced by the left multiplication. In particular, condition (1) in Corollary \ref{cor:WCP to ess free} holds. On the other hand, it is known that the transformation groupoid $\partial_\beta \Gamma \rtimes \Gamma$ has the weak containment property if and only if $\Gamma$ is exact (see \cite{MR3146831, MR3635811}). So Gromov's non-exact groups provide desired examples.
%\end{rem}

\section{Extension of $\tau^\fix_\mu$}\label{sec:reduction}

This section is devoted to discussing when the fixed point trace $\tau^\fix_\mu$ introduced in Section \ref{sec:trace} can be extended to a tracial state on the reduced groupoid $C^*$-algebra $C^*_r(\G)$. For future use, we need a decomposition formula for general tracial states essentially from \cite{CN22, Nes13}.

%In this section, we aim to provide a decomposition formula of a tracial state on the reduced groupoid $C^*$-algebra in terms of states on the reduced group $C^*$-algebras of isotropy fibres. We need to consult the results in \cite{CN22, Nes13} for second-countable groupoids.

Let $\mu$ be a probability Radon measure on $\Gz$. Assume that for $\mu$-almost everywhere $x\in \Gz$, we are given a state $\tau_x$ on $C^*(\G_x^x)$. Denote the canonical unitary generators of $C^*(\G_x^x)$ by $u_\gamma$, for $\gamma \in \G_x^x$. Recall that the field of states $\{\tau_x\}_{x\in \Gz}$ is \emph{$\mu$-measurable} if for any $f \in C_c(\G)$, the function 
\[
\Gz \to \CC, \quad x \mapsto \sum_{\gamma \in \G^x_x} f(\gamma) \tau_x(u_\gamma)
\]
is $\mu$-measurable. 

We recall the following special case of \cite[Theorem 1.1 and Theorem 1.3]{Nes13}:

\begin{prop}[\cite{Nes13}]\label{prop:Nes13 prop 1.1}
Let $\G$ be a locally compact Hausdorff and \'{e}tale groupoid which is second-countable, and $\tau$ be a tracial state on $C^*(\G)$ with the associated measure $\mu$. Then there exists a unique $\mu$-measurable field $\{\tau_x\}_{x\in \Gz}$, where $\tau_x$ is a tracial state on $C^*(\G_x^x)$ such that
\begin{equation}\label{EQ:trace dec}
\tau(f) = \int_{\Gz} \sum_{\gamma \in \G_x^x} f(\gamma)\tau_x(u_\gamma) \d \mu(x) \quad \text{for} \quad f\in C_c(\G). 
\end{equation}
\end{prop}

In the case of $\tau^\fix_\mu$, we can directly calculate its associated measurable field of states as follows. 

\begin{lem}\label{lem:measurable fields for tau fix}
Let $\G$ be a locally compact Hausdorff and \'{e}tale groupoid, and $\mu$ be an invariant probability Radon measure on $\Gz$. For each $x\in \Gz$, we denote by $\tau_x^{\triv}: C_c(\G_x^x) \to \CC$ the trivial representation of $\G_x^x$. Then for every $f \in C_c(\G)$, we have
\begin{equation}\label{EQ:trace dec tau fix}
\tau^\fix_\mu(f) = \int_{\Gz} \tau^\triv_{x}(\eta_x(f)) \d \mu(x),
\end{equation}
where $\eta_x: C_c(\G) \to C_c(\G_x^x), f \mapsto f|_{\G_x^x}$ is the restriction map for each $x\in \Gz$.
\end{lem}

\begin{proof}
It suffices to verify (\ref{EQ:trace dec tau fix}) for $g \in C_c(B)$, where $B$ is a pre-compact open bisection. By definition, $\tau^\triv_{x}(\eta_x(g)) \neq 0$ if and only if there exists $\gamma \in B$ such that $\s(\gamma) = \r(\gamma)=x$ and $g(\gamma) \neq 0$. This happens if and only if $\s(\gamma) \in \fix(\alpha_B)$ and $g(\gamma) \neq 0$. Hence, we obtain
\[
\int_{\Gz} \tau^\triv_{x}(\eta_x(g)) \d \mu(x) = \int_{\fix(\alpha_B)} g((s|_B)^{-1}(x)) \d \mu(x) = \int_{\fix(\alpha_B)} g_B(x) \d \mu(x) = \tau^\fix_\mu(g),
\]
which concludes the proof of (\ref{EQ:trace dec tau fix}). The argument also shows that for such $g\in C_c(B)$, the map $x \mapsto \tau^\triv_{x}(\eta_x(g))=\chi_{\fix(\alpha_B)}(x)\cdot g((s|_B)^{-1}(x))$ is $\mu$-measurable. Hence,  the family $\{\tau^\triv_{x}\}_{x\in \Gz}$ is $\mu$-measurable.
\end{proof}

\begin{rem}\label{rem:2nd countable}
It is worth noticing that we do not need second-countability of $\G$ in the hypothesis of Lemma \ref{lem:measurable fields for tau fix}. The reason is that we can simply find the associated measurable field of tracial states $\{\tau_x^\triv\}_{x\in \Gz}$ and then directly verify (\ref{EQ:trace dec tau fix}) without using Renault’s disintegration theorem. However, Renault’s disintegration theorem was used in the proof of Proposition \ref{prop:Nes13 prop 1.1}.
\end{rem}

\begin{rem}
S. Neshveyev has pointed out to us that the fixed point trace $\tau^\fix_\mu$ in the current paper coincides with $\varphi_\mu''$ considered in \cite[Equation (2.3)]{NS22} by Lemma~\ref{lem:measurable fields for tau fix}.

%Also note that in the case of group actions, $\tau^\fix_\mu$ also appeared in \cite{KTT90}. 
\end{rem}

Proposition \ref{prop:Nes13 prop 1.1} and Lemma \ref{lem:measurable fields for tau fix} deal with tracial states on the maximal groupoid $C^*$-algebra $C^*(\G)$. To study the reduced case, we need to consult the newly developed tool in \cite{CN22}. 
More precisely, we need an exotic $C^*$-norm $\|\cdot\|_e$ on $C_c(\G_x^x)$ for each $x\in \Gz$ introduced in \cite[Definition~2.1]{CN22} which dominates the reduced $C^*$-norm. If we fix $x\in \Gz$, then we have the restriction map
\[
\eta_x: C_c(\G) \to C_c(\G_x^x), \quad f \mapsto f|_{\G_x^x}. 
\]
By \cite[Lemma 1.2]{CN22}, $\eta_x$ extends to completely positive contractions 
\[
\vartheta_x: C^*(\G) \to C^*(\G^x_x) \quad \text{and} \quad \vartheta_{x,r}: C^*_r(\G) \to C^*_r(\G^x_x).
\]
According to \cite[Theorem 2.4]{CN22}, the exotic $C^*$-norm $\|\cdot\|_e$ can be (equivalently) defined as follows: for each $h \in C_c(\G_x^x)$, we define
\[
\|h\|_e:= \inf\{\|f\|_r: f \in C_c(\G) \text{ and } \eta_x(f) = h\}. 
\]
Denote by $C^*_e(\G_x^x)$ the $C^*$-completion of $C_c(\G^x_x)$ with respect to $\|\cdot\|_e$.

\begin{prop}[{\cite[Proposition 3.1]{CN22}}]\label{prop:CN22 Prop 3.1}
Let $\G$ be a locally compact Hausdorff and \'{e}tale groupoid which is second-countable, and  $\tau$ be a tracial state on $C^*(\G)$ with the associated measure $\mu$ on $\Gz$. Then $\tau$ factors through $C^*_r(\G)$ \emph{if and only if} for the associated $\mu$-measurable field of tracial states $\{\tau_x\}_{x\in \Gz}$ (as in Proposition \ref{prop:Nes13 prop 1.1}), $\tau_x$ factors through $C^*_e(\G_x^x)$ (denoted by $\tau_{x,e}$) for $\mu$-almost every $x\in \Gz$. 
\end{prop}

Combining Lemma \ref{lem:measurable fields for tau fix} with Proposition \ref{prop:CN22 Prop 3.1}, we reach the following:

\begin{prop}\label{prop:char for tau fix extension}
Let $\G$ be a locally compact Hausdorff and \'{e}tale groupoid which is second-countable, and $\mu$ be an invariant probability Radon measure on $\Gz$. Then $\tau^\fix_\mu$ extends to a tracial state on $C^*_r(\G)$ \emph{if and only if} the trivial representation $\tau_{x}^\triv$ on $C_c(\G_x^x)$ extends to a tracial state on $C^*_e(\G_x^x)$ for $\mu$-almost every $x\in \Gz$. 
\end{prop}

%\begin{proof}
%Note that the only obstruction from applying Proposition \ref{prop:CN22 Prop 3.1} directly is the hypothesis of $\sigma$-compactness rather than second countability. This can be bypassed using Remark \ref{rem:2nd countable}. More precisely, Lemma \ref{lem:measurable fields for tau fix} shows that (\ref{EQ:trace dec tau fix}) holds for elements in $C_c(\G)$ in the current situation. Using the same argument as in the proof of Proposition \ref{prop:CN22 Prop 3.1}, we can conclude the result merely using $\sigma$-compactness.
%\end{proof}

Consequently, we have the following:

\begin{cor}\label{cor:ess free and faithful}
Let $\G$ be a locally compact Hausdorff and \'{e}tale groupoid which is second-countable, and $\mu$ be an invariant probability Rodan measure on $\Gz$. Then $\G$ is essentially free with respect to $\mu$ \emph{if and only if} the following conditions hold:
\begin{itemize}
 \item $\G^x_x$ is amenable for $\mu$-almost all $x\in \Gz$;
 \item for any tracial state $\tau$ on $C^*_r(\G)$ with associated measure $\mu$, $\tau_{x,e}$ is faithful on $C^*_e(\G^x_x)$ for $\mu$-almost all $x\in \Gz$.
\end{itemize}
\end{cor}

\begin{proof}
By definition, $\G$ is essentially free with respect to $\mu$ if and only if $\G_x^x$ is trivial for $\mu$-almost all $x\in \Gz$, which concludes the proof of the necessity. For sufficiency, we consider the fixed point trace $\tau_\mu^\fix$. By Proposition \ref{prop:char for tau fix extension} and the assumption of amenability of $\G^x_x$ for $\mu$-almost all $x\in \Gz$, $\tau_\mu^\fix$ extends to a tracial state on $C^*_r(\G)$. Moreover for $x\in \Gz$, $(\tau_\mu^\fix)_{x,e}$ is the trivial representation by Lemma \ref{lem:measurable fields for tau fix}. Hence, it is faithful if and only if it is injective, which is equivalent to that $\G_x^x$ is trivial. 
\end{proof}

As another direct corollary of Proposition~\ref{prop:CN22 Prop 3.1}, we also obtain the following:

\begin{cor}\label{cor:CN22 Prop 3.1}
Let $\G$ be a locally compact Hausdorff and \'{e}tale groupoid which is second-countable, and  $\tau$ be a tracial state on $C^*_r(\G)$ with the associated measure $\mu$ on $\Gz$. If $\|\cdot\|_e = \|\cdot\|_r$ on $C_c(\G^x_x)$ for $\mu$-almost every $x\in \Gz$, then for the associated $\mu$-measurable field of tracial states $\{\tau_x\}_{x\in \Gz}$ (as in Proposition \ref{prop:Nes13 prop 1.1}), $\tau_x$ factors through a state $\tau_{x,r}$ on $C^*_r(\G_x^x)$ for $\mu$-almost every $x\in \Gz$. 
\end{cor}

For future use, we record the following decomposition formula extending (\ref{EQ:trace dec}) for general elements in the reduced groupoid $C^*$-algebra $C^*_r(\G)$.

\begin{prop}\label{prop:extend trace dec to general elements}
Let $\G$ be a locally compact Hausdorff and \'{e}tale groupoid which is second-countable, and  $\tau$ be a tracial state on $C^*_r(\G)$ with the associated measure $\mu$ on $\Gz$. Assume that for $\mu$-almost every $x\in \Gz$, we have $\|\cdot\|_e = \|\cdot\|_r$ on $C_c(\G^x_x)$. Then for $a \in C^*_r(\G)$, we have
\begin{equation}\label{EQ:trace dec 2}
\tau(a) = \int_{\Gz} \tau_{x,r}(\vartheta_{x,r}(a)) \d \mu(x),
\end{equation}
where $\tau_{x,r}$ is the associated tracial state on $C^*_r(\G^x_x)$ in Corollary \ref{cor:CN22 Prop 3.1}.
\end{prop}

\begin{proof}
By Proposition \ref{prop:Nes13 prop 1.1} and  Corollary \ref{cor:CN22 Prop 3.1}, we know that (\ref{EQ:trace dec 2}) holds for any $f\in C_c(\G)$. Given $a \in C^*_r(\G)$, we take a sequence $\{f_n\}_{n\in \NN}$ in $C_c(\G)$ such that $\|a-f_n\|_r \to 0$ as $n \to \infty$. For each $x\in \Gz$, we have that
\[
|\tau_{x,r}(\vartheta_{x,r}(f_n)) - \tau_{x,r}(\vartheta_{x,r}(a))| \leq \|\vartheta_{x,r}(f_n-a)\|_r \leq \|f_n - a\|_r \to 0 \quad \text{as} \quad n \to \infty.
\]
Moreover, 
\[
|\tau_{x,r}(\vartheta_{x,r}(f_n))| \leq \|f_n\|_r\leq ||f_n-a||_r+||a||_r \leq 2\|a\|_r
\]
for sufficiently large $n$. Hence, Lebesgue's dominated convergence theorem implies that
\[
\int_{\Gz} \tau_{x,r}(\vartheta_{x,r}(f_n)) \d \mu(x) \to \int_{\Gz} \tau_{x,r}(\vartheta_{x,r}(a)) \d \mu(x) \quad \text{as} \quad n \to \infty.
\]
Since $\tau$ is continuous, we have $\tau(f_n) \to \tau(a)$ as $n \to \infty$. Therefore, we have completed the proof.
\end{proof}

Using an identical argument as in the proof of Proposition \ref{prop:extend trace dec to general elements}, (\ref{EQ:trace dec}) can be always extended from $C_c(\G)$ to $C^*(\G)$ without requiring $\|\cdot\|_e = \|\cdot\|_r$ on $C_c(\G^x_x)$ for $\mu$-almost every $x\in \Gz$. More precisely, we have the following:

\begin{cor}
Let $\G$ be a locally compact Hausdorff and \'{e}tale groupoid which is second-countable, and $\tau$ be a tracial state on $C^*(\G)$ with the associated measure $\mu$. Then there exists a unique $\mu$-measurable field $\{\tau_x\}_{x\in \Gz}$, where $\tau_x$ is a tracial state on $C^*(\G_x^x)$ such that
\begin{equation}
\tau(a) = \int_{\Gz} \tau_x(\vartheta_x(a)) \d \mu(x) \quad \text{for} \quad a \in C^*(\G). 
\end{equation}
\end{cor}

%Now we are in the position to present our partial answer to Question \ref{Q:trace ideal}:
%
%\begin{prop}\label{prop:partial answer to Itau}
%Let $\G$ be a locally compact Hausdorff and \'{e}tale groupoid which is second-countable, and  $\tau$ be a tracial state on $C^*(\G)$ with the associated measure $\mu$ on $\Gz$. Assume that for $\mu$-almost every $x\in \Gz$, we have $\|\cdot\|_e = \|\cdot\|_r$ on $C_c(\G^x_x)$. Then we have $I_\tau \supseteq \tilde{I}(U_{I_\tau})$. Hence in this case, $U_{I_\tau} = V_{I_\tau}$ if and only if $I_\tau = \tilde{I}(V_{I_\tau})$. 
%\end{prop}
%
%\begin{proof}
%By Lemma \ref{lem:U for general tracial ideal}, we have $U_{I_\tau} = \Gz \setminus \supp \mu$, denoted by $U$. Given $a \in \tilde{I}(U)$, we have $a^*a \in \tilde{I}(U)$, and hence $\r(\{\gamma\in \G:j(a^*a)(\gamma) \neq 0\}) \subseteq U$. We aim to show that $\tau(a^*a) = 0$. 
%Note that for any $x \in \Gz \setminus U = \supp \mu$, we have $j(a^*a)|_{\G_x^x} = 0$. Consider the following commutative diagram:
%\[
%\xymatrix{
%		C^*_r(\G) \ar[r]^-{\vartheta_{x,r}} \ar[d]^{j}  & C^*_r(\G_x^x) \ar[d]^{j}  \\
%		C_0(\G) \ar[r]^-{r_x} & C_0(\G^x_x),
%	}
%\]
%where $r_x$ is the restriction map. Hence we obtain that $j(\vartheta_{x,r}(a^*a)) = 0$, which implies that $\vartheta_{x,r}(a^*a) = 0$ since $j$ is injective. By Lemma \ref{lem:extend trace dec to general elements}, we have
%\[
%\tau(a^*a) = \int_{\Gz} \tau_{x,r}(\vartheta_{x,r}(a^*a)) \d \mu(x) = 0,
%\]
%which concludes the proof.
%\end{proof}

\section{Applications}

\subsection{Transformation groupoids}

In this subsection, we apply the results in previous sections to transformation groupoids. If $\Gamma$ is a discrete group acting on a locally compact Hausdorff space $X$, then we form the \emph{transformation groupoid $X \rtimes \Gamma$}. In this case, the reduced groupoid $C^*$-algebra $C^*_r(X \rtimes \Gamma)$ is $*$-isomorphic to the reduced crossed product $C_0(X) \rtimes_r \Gamma$, and the maximal groupoid $C^*$-algebra $C^*(X \rtimes \Gamma)$ is $*$-isomorphic to the maximal crossed product $C_0(X) \rtimes \Gamma$.

For each $\gamma \in \Gamma$, we use the same notation to denote the associated homeomorphism $\gamma: X \ni x \mapsto \gamma x \in X$. Therefore, a probability Radon measure $\mu$ on $X$ is invariant in the sense in Section \ref{ssec:tracial states} if and only if $\gamma$ is $\mu$-preserving for each $\gamma \in \Gamma$. 
%It is easy to see that the notion of invariance defined here is compatible with the one defined in Section \ref{ssec:tracial states} for general groupoids.
 %The proof is straightforward, and hence omitted.

%Given an invariant probability Radon measure $\mu$ on $X$, we say that the action is \emph{essentially free with respect to $\mu$} if $\mu(\fix(\rho_\gamma)) = 0$ for each $\gamma \in \Gamma$. Following Definition \ref{defn:ess free}, we say that the action is \emph{essentially free} is it is essentially free with respect to any invariant probability Radon measure $\mu$ on $X$.
%Note that this is different from the notion of essential freeness defined for the transformation groupoid $X \rtimes \Gamma$ in Section \ref{ssec:principal, free, effective}.

%For each $B \subseteq X$ and $\gamma \neq 1$ in $\Gamma$, the set $B \times \{\gamma\}$ is clearly a bisection and the map $\alpha_{B \times \{\gamma\}}$ defined in (\ref{EQ:alpha B}) is nothing but $\rho_\gamma|_B$. Moreover, we have the following. The proof follows from the inner regularity of $\mu$, and hence omitted. 

The notion of essential freeness for transformation groupoids can be easily simplified as follows:

\begin{lem}\label{lem:char for ess free}
Let $\Gamma$ be a discrete group acting on a locally compact Hausdorff space $X$ with an invariant probability Radon measure $\mu$ on $X$. Then the transformation groupoid $X \rtimes \Gamma$ is essentially free with respect to $\mu$ (in the sense of Definition \ref{defn:ess free}) \emph{if and only if} the action is essentially free with respect to $\mu$ in the sense that $\mu(\fix(\gamma)) = 0$ for every $\gamma \neq 1$ in $\Gamma$.
\end{lem}
\begin{proof}
The necessity follows from the inner regularity of the Radon measure $\mu$, while the sufficiency follows from the fact that every pre-compact bisection in $X \rtimes \Gamma$ is contained in $X\times F$ for some finite $F\subseteq \Gamma$.
\end{proof}

Moreover, one can calculate directly that the fixed point trace $\tau^\fix_\mu$ has the following form:
\begin{equation}\label{EQ:tau fix group action case}
\tau^\fix_\mu: C_c(\Gamma, C_0(X)) \to \CC, \quad \sum_{i=1}^n f_i \gamma_i \mapsto \sum_{i=1}^n \int_{\fix(\gamma_i)} f_i \d \mu,
\end{equation}
where $f_i \in C_0(X)$ and $\gamma_i \in \Gamma$ for $i =1, \dotsc, n$. %as presented in Section \ref{sec:introduction}.

In the maximal crossed products, Theorem \ref{thm:char for ess free max case} recovers \cite[Theorem 2.7]{KTT90}: %and part of \cite[Corollary 2.4]{NS22}:

\begin{prop}\label{prop:char for ess free in group action max}
Let $\Gamma$ be a discrete group acting on a locally compact Hausdorff space $X$ with an invariant probability Radon measure $\mu$ on $X$. 
Then the following are equivalent:
\begin{enumerate}
 \item The action is essentially free with respect to $\mu$;
 \item Any tracial state $\tau$ on $C_0(X) \rtimes \Gamma$ with the associated measure being $\mu$ is canonical.
\end{enumerate}
\end{prop}

In the reduced crossed products, Theorem \ref{thm:char for ess free max case} can be translated as follows:

\begin{prop}\label{prop:char for ess free in group action}
Let $\Gamma$ be a discrete group acting on a locally compact Hausdorff space $X$ with an invariant probability Radon measure $\mu$ on $X$. 
Then the following are equivalent:
\begin{enumerate}
 \item The action is essentially free with respect to $\mu$;
 \item Any tracial state $\tau$ on $C_0(X) \rtimes_r \Gamma$ with the associated measure being $\mu$ is canonical and $\tau^\fix_\mu$ can be extended to a tracial state on $C_0(X) \rtimes_r \Gamma$. 
\end{enumerate}
\end{prop}

Using the discussions in Section \ref{sec:reduction}, we obtain the following:

\begin{lem}\label{lem:char for tau fix extension group action case}
Let $\Gamma$ be a discrete group acting on a locally compact Hausdorff space $X$ with an invariant probability Radon measure $\mu$. Consider the following conditions:
\begin{enumerate}
 \item $\tau^\fix_\mu$ extends to a tracial state on $C_0(X) \rtimes_r \Gamma$;
 \item For $\mu$-almost every $x\in X$, the isotropy group $\Gamma_x:=\{\gamma \in \Gamma: \gamma x = x\}$ is amenable.
\end{enumerate}
Then (2) $\Rightarrow$ (1).

If $\Gamma$ is countable and $X$ is second-countable, then we actually have (1) $\Leftrightarrow$ (2). 
\end{lem}

\begin{proof}
First, we show that ``(2) $\Rightarrow$ (1)''. Recall from Lemma \ref{lem:measurable fields for tau fix} (see also Remark \ref{rem:2nd countable}) that 
\begin{equation*}
\tau^\fix_\mu(f) = \int_{\Gz} \tau^\triv_{x}(\eta_x(f)) \d \mu(x),
\end{equation*}
holds for any $f \in C_c(\G)$. By (2), the trivial representation $\tau^\triv_{x}$ on $C_c(\Gamma_x)$ extends to a state $\tau^\triv_{x,r}$ on $C^*_r(\Gamma_x)$ for $\mu$-almost every $x\in X$. By the same argument as in the proof of Proposition \ref{prop:extend trace dec to general elements} without requiring the second-countability of $X \rtimes \Gamma$ (see Remark \ref{rem:2nd countable}), the map
\[
C_0(X) \rtimes_r \Gamma \longrightarrow \CC, \quad  a \mapsto \int_{\Gz} \tau^\triv_{x,r}(\vartheta_{x,r}(a)) \d \mu(x)
\]
is well-defined and extends $\tau^\fix_\mu$. %By the properties of $\tau^\fix_\mu$, it is clear that this extension provides the required tracial state on $C_0(X) \rtimes_r \Gamma$. 

Now we show that ``(1) $\Rightarrow$ (2)'' if $\Gamma$ is countable and $X$ is second-countable. 
By Proposition \ref{prop:char for tau fix extension}, $\tau^\fix_\mu$ extends to a tracial state on $C_0(X) \rtimes_r \Gamma$ if and only if for $\mu$-almost every $x\in X$, the trivial representation on $C_c((X \rtimes \Gamma)_x^x)$ extends to a state on $C^*_e((X \rtimes \Gamma)_x^x)$. Note that $(X \rtimes \Gamma)_x^x$ is isomorphic to $\Gamma_x$, and it follows from \cite[Proposition 2.10]{CN22} that $\|\cdot\|_e = \|\cdot\|_r$ on $C_c((X \rtimes \Gamma)_x^x)$. It is also known that the trivial representation on $C_c(\Gamma_x)$ extends to a state on $C^*_r(\Gamma_x)$ if and only if $\Gamma_x$ is an amenable group. Thus, we have concluded the proof.
\end{proof}

%\begin{lem}\label{lem:char for tau fix extension group action case}
%Let $\Gamma$ be a countable discrete group acting on a locally compact Hausdorff and second-countable space $X$ equipped with an invariant probability Radon measure $\mu$. Then $\tau^\fix_\mu$ extends to a tracial state on $C_0(X) \rtimes_r \Gamma$ \emph{if and only if} for $\mu$-almost every $x\in X$, the isotropy group $\Gamma_x:=\{\gamma \in \Gamma: \gamma x = x\}$ is amenable.
%\end{lem}
%
%\begin{proof}
%By Proposition \ref{prop:char for tau fix extension}, $\tau^\fix_\mu$ extends to a tracial state on $C_0(X) \rtimes_r \Gamma$ if and only if for $\mu$-almost every $x\in X$, the trivial representation on $C_c((X \rtimes \Gamma)_x^x)$ extends to a state on $C^*_e((X \rtimes \Gamma)_x^x)$. Note that $(X \rtimes \Gamma)_x^x$ is isomorphic to $\Gamma_x$, and it follows from \cite[Proposition 2.10]{CN22} that $\|\cdot\|_e = \|\cdot\|_r$ on $C_c((X \rtimes \Gamma)_x^x)$. Also it is known that the trivial representation on $C_c(\Gamma_x)$ extends to a state on $C^*_r(\Gamma_x)$ if and only if $\Gamma_x$ is amenable. Hence we conclude the proof.
%\end{proof}

Combining the previous lemma with Proposition \ref{prop:char for ess free in group action}, we have obtained the main result of this subsection:

\begin{cor}\label{cor:char for tau fix extension group action case}
Let $\Gamma$ be a discrete group acting on a locally compact Hausdorff space $X$ with an invariant probability Radon measure $\mu$. We consider the following conditions:
\begin{enumerate}
 \item The action is essentially free with respect to $\mu$;
 \item Any tracial state $\tau$ on $C_0(X) \rtimes_r \Gamma$ with the associated measure being $\mu$ is canonical;
 \item The isotropy group $\Gamma_x$ is amenable for $\mu$-almost every $x\in X$.
\end{enumerate}
Then (1) $\Rightarrow$ (2) and (2) + (3) $\Rightarrow$ (1). In particular, if the action is essentially free then any tracial state $\tau$ on $C_0(X) \rtimes_r \Gamma$ is canonical.

If additionally $\Gamma$ is countable and $X$ is second-countable, then (1) $\Leftrightarrow$ (2) + (3). 
\end{cor}

%{\color{red}{Hence if the action is not essentially free (\emph{e.g.}, the example constructed in \cite{Jos23}), then there are more than one tracial states on the reduced crossed product $C_0(X) \rtimes_r \Gamma$. }}

\begin{rem}\label{rem:C*-simple}
Note that ``(2) $\Rightarrow$ (1)'' in Corollary \ref{cor:char for tau fix extension group action case} does not hold in general. Indeed, it was shown in \cite[Corollary 1.4]{BK18} that if the amenable radical of $\Gamma$ (\emph{i.e.}, the largest amenable normal subgroup of $\Gamma$) is trivial (\emph{e.g.}, $\Gamma$ is $C^*$-simple), then any tracial state on the reduced crossed product is canonical. 
Actually, it was shown in \cite[Corollary 1.12]{Urs21} that any tracial state on the reduced crossed product is canonical if and only if the action of the amenable radical of $\Gamma$ on $X$ is essentially free. On the other hand, trivial actions by non-trivial groups cannot be essentially free.
%Moreover, Corollary~\ref{cor:char for tau fix extension group action case} and \cite[Corollary 1.4]{BK18} imply that if the largest amenable normal subgroup of the group $\Gamma$ is trivial and all the isotropy groups are amenable, then the action 
%In particular, Question \ref{Ques:converse} has a negative answer for any non-trivial $C^*$-simple group acting trivially on a space.
\end{rem}

Finally, we focus on the case of trivial actions. In this case, each measure on $X$ is invariant. Moreover, we have 
\[
\tau^\fix_\mu(\sum_{i=1}^n f_i \gamma_i) = \sum_{i=1}^n \int_{\fix(\gamma_i)} f_i \d \mu = \sum_{i=1}^n \int_X f_i \d \mu, \quad \text{for} \quad \sum_{i=1}^n f_i \gamma_i \in C_c(\Gamma, C_0(X)).
\]
This shows that on $C_c(\Gamma, C_0(X)) \cong C_c(\Gamma) \otimes C_0(X)$ we have 
\[
\tau^\fix_\mu = \tau^\triv \otimes \tau_\mu, 
\]
where $\tau^\triv$ is the trivial representation of $C_c(\Gamma)$, and $\tau_\mu(f) = \int_X f \d\mu$ for $f \in C_0(X)$. It is worth noticing that $\tau_\mu$ is the same as (\ref{EQ:tau mu}) in the special case when $\G=X$. 

%Therefore, we obtain the following:

%Recall that the trivial representation $\tau_\triv$ on $C_c(\Gamma)$ is defined by $\tau_\triv(\sum_{i=1}^n a_i \gamma_i) = \sum_{i=1}^n a_i$ for $\sum_{i=1}^n a_i \gamma_i \in C_c(\Gamma)$. Hence the calculation above shows that on $C_c(\Gamma, C_0(X)) \cong C_c(\Gamma) \otimes C_0(X)$, we have $\tau^\fix = \tau_\triv \otimes \tau_\mu$ where $\tau_\mu(f) = \int_X f \d\mu$ for $f \in C_0(X)$. Therefore, we obtain the following:

\begin{lem}\label{lem:trivial action extension}
Let $\Gamma$ be a discrete group acting trivially on a compact Hausdorff space $X$ with a probability Radon measure $\mu$. Then the following are equivalent:
\begin{enumerate}
 \item $\tau^\fix_\mu$ can be extended to a tracial state on $C(X) \rtimes_r \Gamma$;
 \item $\Gamma$ is amenable;
 \item $X \rtimes \Gamma$ has the weak containment property.
\end{enumerate}
In particular, if $\Gamma$ is not amenable then $\tau^\fix_\mu$ cannot be extended to a tracial state on $C(X) \rtimes_r \Gamma$.
\end{lem}

\begin{proof}
``(1) $\Leftrightarrow$ (2)'': In this case, we already know that $\tau^\fix_\mu = \tau^\triv \otimes \tau_\mu$ and $C(X) \rtimes_r \Gamma \cong C^*_r(\Gamma) \otimes C(X)$. So $\tau^\fix_\mu$ can be extended to a tracial state on $C(X) \rtimes_r \Gamma$ if and only if $\tau^\triv$ can be extended to a tracial state on $C^*_r(\Gamma)$, which is well-known to be equivalent to amenability of $\Gamma$.

``(2) $\Leftrightarrow$ (3)'': If $\Gamma$ is amenable, then $X \rtimes \Gamma$ is amenable and hence has the weak containment property. Conversely, if we assume that $X \rtimes \Gamma$ has the weak containment property, then the canonical quotient map $C(X)\otimes_{\text{max}}C^*(\Gamma)=C(X) \rtimes \Gamma \to C(X) \rtimes_r \Gamma=C(X)\otimes_{\text{min}}C_r^*(\Gamma)$ is injective, which implies that $C^*(\Gamma) \to C^*_r(\Gamma)$ is injective by a diagram chase. Therefore, $\Gamma$ is amenable. 
\end{proof}

We note that trivial actions by non-trivial groups cannot be essentially free. Hence, Proposition \ref{prop:char for ess free in group action} and Lemma \ref{lem:trivial action extension} together imply the following:

\begin{cor}\label{cor:trivial action}
Let $\Gamma$ be a non-trivial discrete group acting trivially on a compact Hausdorff space $X$ with a probability Radon measure $\mu$. Then at least one of the following conditions \emph{fails}:
\begin{enumerate}
 \item $\Gamma$ is amenable;
 \item Any tracial state $\tau$ on $C(X) \rtimes_r \Gamma$ with the associated measure being $\mu$ is canonical.
\end{enumerate}
\end{cor}

\subsection{Tracial ideals and quasidiagonal traces}
%By an \emph{ideal} in a $C^*$-algebra, we mean a closed two-sided ideal. 

Given a tracial state on the reduced groupoid $C^*$-algebra, we can consider its tracial ideal as follows.

\begin{defn}\label{defn:tracial ideal}
Let $\G$ be a locally compact Hausdorff and \'{e}tale groupoid, and $\tau$ be a tracial state on $C^*_r(\G)$. The \emph{tracial ideal associated to $\tau$} is defined to be $I_\tau:=\{a \in C^*_r(\G): \tau(a^*a) = 0\}$, which is a closed two-sided ideal in $C^*_r(\G)$. For an invariant probability Radon measure $\mu$ on $\Gz$, we simply write $I_\mu:=I_{\tau_\mu}$, where $\tau_\mu$ is the tracial state associated to $\mu$ defined in (\ref{EQ:tau mu}). %Also call $I_\mu$ the \emph{associated tracial ideal} of $\mu$.
\end{defn}

\begin{rem}\label{rem:I tau}

For a tracial state $\tau$, we always have $I_\tau \subseteq \ker (\tau)$. Indeed, let $a$ be any positive element in $I_\tau$. If $a= b^*b$ for some $b\in I_\tau$, then by definition we have $0 = \tau(b^*b) = \tau(a)$. 
\end{rem}

The following proposition is the key observation in this subsection:

\begin{prop}\label{prop:exact sequence}
Let $\G$ be a locally compact Hausdorff and \'{e}tale groupoid, and $\mu$ be an invariant probability Radon measure on $\Gz$. Then we have the following short exact sequence:
\[
0 \longrightarrow I_\mu \longrightarrow C^*_r(\G) \longrightarrow C^*_r(\G_{\supp \mu}) \longrightarrow 0.
\]
\end{prop}

\begin{proof}
Firstly, we note that $\mu$ being invariant implies that $\supp \mu$ is invariant. According to \cite[Proposition 5.4]{BFPR22}, it suffices to show that
\[
I_\mu = \{a \in C^*_r(\G): E(a^*a) \in C_0(\Gz \setminus \supp \mu)\}. 
\]
Given $a\in C^*_r(\G)$, we note that $E(a^*a)$ is a continuous non-negative function on $\Gz$. Hence, it is clear that $\int_{\Gz} E(a^*a) \d\mu = 0$ if and only if $\{x: E(a^*a)(x) \neq 0\} \subseteq \Gz \setminus \supp \mu$. This concludes the proof.
\end{proof}

\begin{cor}\label{quotientgrou}
Let $\G$ be a locally compact Hausdorff and \'{e}tale groupoid which is essentially free. Then for any tracial state $\tau$ on $C^*_r(\G)$, we have $C^*_r(\G)/I_\tau \cong C^*_r(\G_{\supp \mu_\tau})$.
\end{cor}
\begin{proof}
It follows directly from Theorem~\ref{thm:char for ess free max case} and Proposition~\ref{prop:exact sequence}.
\end{proof}

A question of N. Brown asks whether every amenable tracial state is quasidiagonal (see \cite[Question 6.7(2)]{Bro06}) in the following sense:
\begin{defn}
A tracial state $\tau$ on a $C^*$-algebra $A$ is \emph{quasidiagonal} if there is a net of contractive completely positive maps $\phi_i: A \rightarrow M_{k(i)}$ such that
\begin{itemize}
\item $\tau(a)=\lim_i \text{tr}\circ \phi_i (a)$ for all  $a\in A$;

\item $\lim_i ||\phi_i(ab)-\phi_i(a)\phi_i(b)||=0$ for all $a,b\in A$.
\end{itemize}
\end{defn}

 Substantial progress on this question has recently been made in \cite{MR3579134, MR3583354} as follows:
\begin{thm}[see \cite{MR3583354,MR3579134}]\label{TWWquasi}
Any faithful, amenable tracial state on a separable, exact $C^*$-algebra satisfying the UCT is quasidiagonal.
\end{thm}
We refer the reader to \cite{Bro06, MR3583354,MR3579134} for relevant definitions and the question of N. Brown. We end this paper by removing the condition ``faithful'' on reduced groupoid $C^*$-algebras. More precisely, we prove the following result:
\begin{thm}\label{quasiThm}
Let $\G$ be a locally compact, Hausdorff, second-countable and étale groupoid such that $C_r^*(\G)$ is an exact $C^*$-algebra and $\G$ satisfies the strong Baum--Connes conjecture with all coefficients in the sense of \cite[Definition~3.6]{BP23}. If $\G$ is also essentially free, then every amenable tracial state on $C_r^*(\G)$ is quasidiagonal.
\end{thm}
\begin{proof}
Let $\tau$ be any amenable tracial state on $C_r^*(\G)$. Then $\tau$ vanishes on the tracial ideal $I_\tau$ by Remark~\ref{rem:I tau}, and hence there is an induced faithful tracial state $\dot{\tau}$ on $C_r^*(\G)/I_\tau$ such that $\tau=\dot{\tau}\circ p$, where $p:C_r^*(\G)\rightarrow C_r^*(\G)/I_\tau$ is the quotient map. To see that $\tau$ is quasidiagonal, it suffices to show that $\dot{\tau}$ is quasidiagonal.

Since $\tau$ is amenable and $C_r^*(\G)$ is exact, it follows that $\dot{\tau}$ is also amenable (see \cite[Proposition 6.3.5 (4)]{MR2391387}). By Corollary~\ref{quotientgrou}, there is a closed invariant subset $D$ of $\G^{(0)}$ such that $C_r^*(\G)/I_\tau\cong C_r^*(\G_D)$, which is exact as quotients of exact $C^*$-algebras are exact by \cite[Theorem~10.2.5]{MR2391387}. As $\G$ is second-countable and satisfies the strong Baum--Connes conjecture with all coefficients, $C_r^*(\G_D)$ is separable and satisfies the UCT by \cite[Theorem~4.11 and Corollary~4.2]{BP23}. Thus, we conclude the quasidiagonality of $\dot{\tau}$ from Theorem~\ref{TWWquasi}.
\end{proof}
As a consequence, we obtain the following corollary (see a similar result in \cite[Remark~3.13]{BL20}):
\begin{cor}\label{cor:amentrace}
Let $\G$ be a locally compact, Hausdorff, second-countable, amenable and étale groupoid, which is also essentially free. Then every tracial state on $C_r^*(\G)$ is quasidiagonal.
\end{cor}
\begin{proof}
If $G$ is amenable, then $\G$ satisfies the strong Baum--Connes conjecture with all coefficients by \cite[Corollary~3.15]{BP23} and $C_r^*(G)$ is nuclear (hence also exact). As every tracial state on a nuclear $C^*$-algebra is amenable \cite[Proposition~6.3.4]{MR2391387}, we complete the proof by Theorem~\ref{quasiThm}.
\end{proof}

\textbf{Acknowledgement}: We wish to thank Johannes Christensen, Sergey Neshveyev and Eduardo Scarparo for helpful discussions on this topic. Finally, we would like to thank the anonymous reviewers for the helpful comments.

\bibliographystyle{plain}

\bibliography{Ghostly_ideals}

% \printbibliography

\end{document}